\documentclass[11pt]{article}
\usepackage[numbers,sort&compress]{natbib}
\usepackage{enumerate}
\usepackage{amscd}
\usepackage{amsmath}
\usepackage{latexsym}
\usepackage{amsfonts}
\usepackage{setspace}
\usepackage{amssymb}
\usepackage{amsthm}
\usepackage{verbatim}
\usepackage{mathrsfs}
\usepackage{enumerate}
\usepackage[hypertexnames=false]{hyperref}

\theoremstyle{plain}
\theoremstyle{definition}\newtheorem{theorem}{Theorem}[section]
\theoremstyle{definition}\newtheorem{lemma}[theorem]{Lemma}
\theoremstyle{plain}
\theoremstyle{plain}
\theoremstyle{definition}\newtheorem{remark}{Remark}[section]
\usepackage{xcolor}

\newcommand{\mr}{\mathbb{R}}
\newcommand{\mt}{\mathbb{T}}
\newcommand{\mz}{\mathbb{Z}}

\newcommand{\ben}{\begin{enumerate}}
	\newcommand{\een}{\end{enumerate}}

\newcommand{\ii}{\mathrm{i}}
\newcommand{\e}{\mathrm{e}}
\newcommand{\dd}{\mathrm{~d}}
\newcommand{\w}{\widehat}
\newcommand{\tb}{\tilde{b} }

\newcommand{\Rmnum}[1]{\expandafter\@slowromancap\romannumeral #1@}

\allowdisplaybreaks

\numberwithin{equation}{section}
\begin{document}
	\title{Sharp decay estimates and asymptotic stability for incompressible MHD equations without viscosity or magnetic diffusion}
	\author{Yaowei Xie\footnote{School of Mathematical Sciences, Capital Normal University, Beijing, 100048, P. R. China. Email: mathxyw@163.com}~~~\,\,\,\,Quansen Jiu\footnote{School of Mathematical Sciences, Capital Normal University, Beijing, 100048, P. R. China. Email: jiuqs@cnu.edu.cn}~~~\,\,\,\,Jitao Liu\footnote{Department of Mathematics, School of Mathematics, Statistics and Mechanics, Beijing University of Technology, Beijing, 100124, P. R. China. Emails: jtliu@bjut.edu.cn,\,\,jtliumath@qq.com}}
\date{}
\maketitle
\begin{abstract}
Whether the global existence and uniqueness of strong solutions of $n$-dimensional incompressible magnetohydrodynamic ({\it MHD for short}) equations with only kinematic viscosity or magnetic diffusion holds true or not remains an outstanding open problem.
In recent years, stared from the pioneer work by Lin and Zhang {\it [Comm. Pure Appl. Math. 67 (2014), no. 4, pp.531--580]},
much more attention has been paid to the case when the magnetic field close to an equilibrium state ({\it the background magnetic field for short}). Specifically, when the background magnetic field satisfies the Diophantine condition (see \eqref{Diophantine} for details), Chen, Zhang and Zhou {\it [Sci. China Math. 41 (2022), pp.1-10]} first studied the perturbation system and established the decay estimates and asymptotic stability of its solutions in 3D periodic domain $\mathbb{T}^3$, which was then
improved to $H^{(3+2\beta)r+5+(\alpha+2\beta)}(\mathbb{T}^2)$ for 2D periodic domain $\mathbb{T}^2$ and any $\alpha>0$, $\beta>0$ by
Zhai {\it [J. Differential Equations 374 (2023), pp.267-278]}. In this paper, we seek to find the optimal decay estimates and improve the space where the global stability is taking place.
Through deeply exploring and effectively utilizing the structure of perturbation system, we discover a {\it new} dissipative mechanism, which enables us to establish the decay estimates in the Sobolev spaces with {\it much lower} regularity. Based on the above discovery, we {\it greatly} reduce the initial regularity requirement of aforesaid two works from $H^{4r+7}(\mathbb{T}^3)$ and $H^{(3+2\beta)r+5+(\alpha+2\beta)}(\mathbb{T}^2)$ to $H^{(3r+3)^+}(\mathbb{T}^n)$ for $r>n-1$ when $n=3$ and $n=2$ respectively. Additionally, we first present the linear stability result via the method of spectral analysis in this paper. From which, the decay estimates obtained for the nonlinear system can be seen as {\it sharp} in the sense that they are in line with those for the linearized system.
\end{abstract}
	\noindent {\bf MSC(2020):}\quad 35A01, 35B35, 35Q35, 76E25, 76W05.
    \vskip 0.02cm
	\noindent {\bf Keywords:} Incompressible MHD equations, global well-posedness, stability,\\ decay estimates.
	\section{Introduction}

\,\,\,\,\,\,\,\,In this paper, we study the $n$-dimensional incompressible MHD equations in the periodic domain $\mathbb{T}^n$:
\begin{align}\label{mhd}
	\left\{\begin{array}{l}
		\partial_t u-\mu\Delta u+u \cdot \nabla u+\nabla p=b\cdot \nabla b, \\
		\partial_t b-\nu\Delta b+u\cdot \nabla b=b\cdot\nabla u, \\
		\nabla\cdot u=\nabla \cdot b=0, \\
		u(x, 0)=u_0(x),\,\, b(x, 0)=b_0(x),
	\end{array}\right.
\end{align}
where $u(x,t)$, $b(x,t)$ and $p(x,t)$ denote velocity field, magnetic field and the pressure separately, nonnegative constants
$\mu$ and $\nu$ are the kinematic viscosity and magnetic diffusion separately.

It is well-known that the MHD equations accurately model numerous significant phenomena, such as magnetic reconnection in astrophysics and geomagnetic dynamo in geophysics (\cite{priest-2000,roberts1967introduction}). Due to this, there have been numerous works dedicated to investigating the global well-posedness of \eqref{mhd}, we refer the readers to \cite{sermange-1983,duvaut-1972}. When $\mu$ and $\nu$ are both positive, the global strong
solution is shown to exist uniquely in 2D space and the corresponding problem in 3D case is still open.

However, if either $\mu$ or $\nu$ is zero, even in 2D space, the global well-posedness will become even more challenging and remains open.
Most recent research is focused on the case $\mu=0$, $\nu=1$ in the periodic domain $\mathbb{T}^2$. In \cite{zhou-2018-GlobalClassicalSolutions}, Zhou and Zhu  obtained the global existence and uniqueness of small solution by assuming that the initial data have reflection symmetry. Subsequently, Wei and Zhang in \cite{zhang-2020-GlobalWellPosedness2D} also got the global well-posedness of small solution in $H^4(\mathbb{T}^2)$ through removing the reflection symmetry assumption and adding mean-zero assumption on initial magnetic field, the readers can also refer to \cite{ye-2022-GlobalWellposednessNonviscous} for related study in the framework of Besov space.

Stared from Lin and Zhang \cite{lin-2014-GlobalSmallSolutions}, for the case $\mu=1$ and $\nu=0$, there have been massive researches devoted to the global well-posedness of \eqref{mhd} under a strong background magnetic field $\overline{b}=(0,1)$ or $\overline{b}=(0,0,1)$, such as \cite{lin-2015-GlobalSmallSolutions,panGlobalClassicalSolutions2018,zhangting,lin-2014-GlobalSmallSolutions,abidi-2017-GlobalSolution3D,jiang-2021-Asymptotic behaviors,xu-2015-GlobalSmallSolutions,ren2014global} and the references therein. In addition to $\overline{b}$, there is another constant state $\tb$ satisfying the following Diophantine condition, which holds for almost $\tb\in \mr^n$.
\newline{\bf Diophantine condition}: For every non zero vector $k \in \mathbb{Z}^n$ and $r>n-1$, there exists a constant $c>0$ such that
\begin{align}\label{Diophantine}
	|\tilde{b}\cdot k|\geq \frac{c}{|k|^r}.
\end{align}
A natural question is: what does the Diophantine condition refers to and which constants satisfy it\,? In fact, as demonstrated in \cite{Alinhac} and \cite{chen-2022-3dmhd-Diophant},
the Diophantine condition is satisfied for almost all vector fields $\tb\in\mr^n$. For convenience of the readers, we will provide its proof as follows.
\newline
$~~~~~${\it For any $\varepsilon>0$ and ball $B$, we set $E_{k,\varepsilon}=\{\tb\in B:|\tb\cdot k|\leq\frac{\varepsilon}{|k|^r}\}$ and it is not hard to conclude $|E_{k,\varepsilon}|\leq \frac{C\varepsilon}{|k|^{r+1}}$ and hence
\begin{equation}\label{D11}
 \left|\mathop{\cup}\limits_{k\in\mz^n\backslash\{0\}}E_{k,\varepsilon}\right|\leq C\varepsilon\sum_{k\in\mz^n\backslash\{0\}}\frac{1}{|k|^{r+1}}\leq C\varepsilon,
\end{equation}
due to $r>2$. \eqref{D11} shows
\begin{equation}\label{D22}
 \left|\mathop{\cap}\limits_{\varepsilon}\mathop{\cup}\limits_{k\in\mz^n\backslash\{0\}}E_{k,\varepsilon}\right|=0,
\end{equation}
which concludes the conclusion.}

In \cite{chen-2022-3dmhd-Diophant}, Chen, Zhang and Zhou first studied the global well-posedness and stability for incompressiblle MHD equations
when the magnetic field close to the equilibrium state $\tb$ satisfying the Diophantine condition in 3D periodic domain $\mathbb{T}^3$. When the magnetic field close to $\tb$, it suffices to consider the following perturbation system by using $b$ to represent $b - \tilde{b}$,
\begin{align}\label{equation}
	\left\{\begin{array}{l}
		\partial_t u-\mu\Delta u+u \cdot \nabla u+\nabla p=\tilde{b} \cdot \nabla b+b \cdot \nabla b, \\
		\partial_t b-\nu\Delta b+u \cdot \nabla b=\tilde{b} \cdot \nabla u+b \cdot \nabla u, \\
		\nabla\cdot u=\nabla \cdot b=0, \\
		u(x, 0)=u_0(x),\,\, b(x, 0)=b_0(x).
	\end{array}\right.
\end{align}
Following is a precise statement of their result.

{\bf Theorem [Chen-Zhang-Zhou,\cite{chen-2022-3dmhd-Diophant}].} {\it Consider the MHD system \eqref{equation} with $\mu=1$, $\nu=0$ or $\mu=0$, $\nu=1$. Assume that $(u_0,b_0)\in H^m(\mathbb{T}^3)$ for $m\geq 4r+7$ satisfies
$$\|u_0\|_{H^m(\mathbb{T}^3)}+\|b_0\|_{H^m(\mathbb{T}^3)}\leq\varepsilon,$$
and
$$\int_{\mathbb{T}^3}u_0\dd x=\int_{\mathbb{T}^3}b_0\dd x=0.$$
If $\varepsilon$ is small enough, for any $t\in[0,+\infty)$, there exists a unique global solution $(u,b)\in C([0,+\infty);H^m(\mathbb{T}^3))$ to the system \eqref{equation}}
satisfying
\begin{equation}\label{de1}
 \|u(t)\|_{H^{r+4}(\mathbb{T}^3)}+\|b(t)\|_{H^{r+4}(\mathbb{T}^3)}\leq C(1+t)^{-\frac{3}{2}},
\end{equation}
and
$$\|u\|_{H^m(\mathbb{T}^3)}+\|b\|_{H^m(\mathbb{T}^3)}\leq C\varepsilon.$$
Later on, for 2D periodic domain $\mathbb{T}^2$, Zhai \cite{zhai-2023-2dmhdstability-Diophant} obtained the asymptotic stability result in $H^m(\mathbb{T}^2)$ with $m\geq{(3+2\beta)r+5+(\alpha+2\beta)}$
for any $\alpha>0$ and $\beta>0$. It should be noted that, compared with \cite{chen-2022-3dmhd-Diophant}, Zhai's result reduce the regularity requirement on initial data slightly. The key point lies in that more effective decay estimates in Sobolev space with {\it lower regularity} are established, namely
\begin{equation}\label{de2}
 \|u(t)\|_{H^{r+3+\alpha}(\mathbb{T}^2)}+\|b(t)\|_{H^{r+3+\alpha}(\mathbb{T}^2)}\leq C(1+t)^{-(1+\beta)}.
\end{equation}
Afterwards, we found another interesting work by Jiang and Jiang \cite{jiang-2023-arma-nonresistive} for the 3D model without magnetic diffusion. For the background magnetic field satisfying the Diophantine condition and by setting $r=3$, they proved the stability of solutions for $(u_0,B_0)\in H^{17}(\mathbb{T}^3)\times H^{21}(\mathbb{T}^3)$ with some classes of large initial perturbation of magnetic field in Lagrangian coordinates. Recently, the issue on stability of solutions for compressible MHD equations is also studied in \cite{WuZhai} and \cite{lizhai-2023-jga}.

{\it At this moment, the following questions naturally arise}:
\begin{enumerate}
  \item How about the linear stability of system \eqref{equation}\,?
  \item Whether the regularity index $s=r+3+\alpha$ in \eqref{de2} is {\it optimal}\,?
  \item What is {\it optimal} decay estimate of $u$ and $b$ in $H^s(\mathbb{T}^n)$\,?
  \item What is the {\it lowest} regularity requirement on initial data such that the system \eqref{equation} is global well-posed\,?
\end{enumerate}

Motivated by above four questions, in this paper, we seek to their answers. For a deeper understanding of system \eqref{equation} and answering question {\it (i)}, we first investigate the following linearized system of \eqref{equation}
\begin{align}\label{lineartheorem}
	\left\{\begin{array}{l}
		\partial_t U-\mu\Delta U=\tilde{b} \cdot \nabla B, \\
		\partial_t B-\nu\Delta B=\tilde{b} \cdot \nabla U,\\
        \nabla\cdot U=\nabla \cdot B=0, \\
		U(x, 0)=U_0(x), B(x, 0)=B_0(x),
	\end{array}\right.
\end{align}
via the method of spectral analysis and obtain the linear stability result as follows.

\begin{theorem}\label{thm1}
	Let $n\geq 2$, $r>n-1$ and $\mu=1$, $\nu=0$ or $\mu=0$, $\nu=1$. Assume that $\left(U_0, B_0\right)\in H^s(\mt^n)$ with $s\geq 0$, and satisfies
\begin{align}\label{mean-zero}
		\int_{\mathbb{T}^n}U_0\dd x=0 \mathrm{~and~} \int_{\mathbb{T}^n}B_0\dd x=0.
	\end{align}
Let $(U,B)$ be the corresponding solution of \eqref{lineartheorem}, then the estimates
\begin{align}\label{lineark}
\|U\|_{H^\alpha(\mathbb{T}^n)}+\|B\|_{H^\alpha(\mathbb{T}^n)}\leq  C (1+t)^{-\frac{s-\alpha}{2(1+r)}}\left(\|U_0\|_{H^s(\mathbb{T}^n)}+\|B_0\|_{H^s(\mathbb{T}^n)}\right),
\end{align}
holds for any $0\leq \alpha\leq s $.
\end{theorem}

\begin{remark}\label{rmk4}
Theorem \ref{thm1} is the first linear stability result for the incompressible MHD equations with only kinematic viscosity or magnetic diffusion
when the magnetic field close to the equilibrium state satisfying the Diophantine condition on periodic domain.
\end{remark}

To establish the linear stability result, we begin with exploiting the wave structure hidden in the system \eqref{equation}. The discovery here helps us to represent the corresponding
wave equations (see \eqref{linear}) in an integral form (see \eqref{new1} and \eqref{new2} for details). The key components of this representation are several kernel operators given by
the Fourier multipliers, which are anisotropic and inhomogeneous. Through estimating the kernel operators in three different frequency spaces precisely, we successfully establish corresponding
linear stability analysis and large time decay estimates.

Through deeply exploring and fully utilizing the structure of perturbation system \eqref{equation}, we successfully obtain our second result about {\it global well-posedness, large time behavior} and {\it nonlinear stability} of system \eqref{equation}.

\begin{theorem}\label{thm}
	Let $n\geq 2$, $r>n-1$ and $\mu=1$, $\nu=0$ or $\mu=0$, $\nu=1$. Assume that $(u_0,b_0)\in H^m(\mt^n)$ for $m>3(r+1)$ satisfying
	\begin{align}\label{initial}
		\int_{\mathbb{T}^n}u_0\dd x=\int_{\mathbb{T}^n}b_0\dd x=0,
	\end{align}
and there is a small constant $\varepsilon$ so that
\begin{align}\label{smallcondition}
	\|u_0\|_{H^m(\mt^n)}+\|b_0\|_{H^m(\mt^n)}\leq \varepsilon.
\end{align}
Then there exists a unique and global solution $(u,b)\in C([0,\infty); H^m(\mt^n))$ of incompressible MHD equations \eqref{equation} such that for any $t\geq 0$,
\begin{align}\label{smallcondition2}
	\|u(t)\|_{H^m(\mt^n)}+\|b(t)\|_{H^m(\mt^n)}\leq C\varepsilon,
\end{align}
and
\begin{align}
	\|u(t)\|_{H^\alpha(\mt^n)}+\|b(t)\|_{H^\alpha(\mt^n)}\leq C(1+t)^{-\frac{m-\alpha}{2(r+1)}},\quad{\rm for\,\,\,any}\,\,r+1\leq \alpha\leq m\label{finaldecay}.
\end{align}
\end{theorem}

\begin{remark}\label{rmk3}
Theorem \ref{thm} {\it greatly} reduces the regularity requirement on the initial data of \cite{chen-2022-3dmhd-Diophant} and \cite{zhai-2023-2dmhdstability-Diophant} from $H^{4r+7}(\mathbb{T}^3)$
and $H^{(3+2\beta)r+5+(\alpha+2\beta)}(\mathbb{T}^2)$ to $H^{3(r+1)^+}(\mathbb{T}^n)$ for $r>n-1$, $\alpha>0$ and $\beta>0$ when $n=3$ and $n=2$ respectively.
\end{remark}

\begin{remark}\label{rmk1}
Compared with \eqref{de1} and \eqref{de2}, in the Sobolev spaces with {\it much lower} regularity $H^{r+1}(\mt^n)$, we obtain the {\it effective} decay estimate.
\end{remark}

\begin{remark}\label{rmk2}
The decay estimates obtained in \eqref{finaldecay} are {\it sharp} for the nonlinear system \eqref{equation} in the sense that they are in line with those for the linearized system \eqref{lineark}.
\end{remark}

\begin{remark}\label{rmk5}
Theorem \ref{thm} holds for any dimensions $n$, and the differences between $n$ are reflected in the initial regularity requirement and decay estimates \eqref{finaldecay}, because both of them are determined by the parameter $r+1$, where $r+1>n$.
\end{remark}

\begin{remark}\label{rmk7}
As pointed out in \cite{Alinhac} and \cite{chen-2022-3dmhd-Diophant}, the vectors consists of {\it rational numbers} or {\it at least one zero} are not contained in the Diophantine condition. Thus, our results are not contained in \cite{panGlobalClassicalSolutions2018,zhou-2018-GlobalClassicalSolutions,zhang-2020-GlobalWellPosedness2D} concerning the global well-posedness of small solution or solution close to the strong background magnetic field $\overline{b}=(0,...,0,1)\in\mr^n$ on periodic domain.
\end{remark}

\begin{remark}\label{rmk6}
The proofs for the case $\mu=1$, $\nu=0$ are quite similar with the case $\mu=0$, $\nu=1$. In fact, it suffices to modify Lemma \ref{lem3} by replacing
\begin{align*}
	-\frac{\dd}{\dd t}\int_{\mathbb{T}^n}\left(\tb\cdot\nabla u\right)\cdot \Lambda^s b\dd x&\leq \frac{3}{2}\left\|\Lambda^{\frac{s}{2}+2}b\right\|_{L^2}-\frac{1}{2}\left\|\Lambda^{\frac{s}{2}}\left(\tb\cdot\nabla u\right)\right\|_{L^2}\\
	&~~~~+\left(\|\nabla b\|_{L^\infty}+\|\nabla u\|_{L^\infty}\right) \left(\|\Lambda^{\frac{s}{2}+1} b\|_{L^2}^2+\|\Lambda^{\frac{s}{2}+1} u\|_{L^2}^2\right),
\end{align*}
with
\begin{align*}
	-\frac{\dd}{\dd t}\int_{\mathbb{T}^n}\left(\tb\cdot\nabla b\right)\cdot \Lambda^s u\dd x&\leq \frac{3}{2}\left\|\Lambda^{\frac{s}{2}+2}u\right\|_{L^2}-\frac{1}{2}\left\|\Lambda^{\frac{s}{2}}\left(\tb\cdot\nabla b\right)\right\|_{L^2}\\
	&~~~~+\left(\|\nabla b\|_{L^\infty}+\|\nabla u\|_{L^\infty}\right) \left(\|\Lambda^{\frac{s}{2}+1} b\|_{L^2}^2+\|\Lambda^{\frac{s}{2}+1} u\|_{L^2}^2\right).
\end{align*}
To avoid duplication, we only provide the proofs for the case $\mu=0$, $\nu=1$ in Subsection \ref{4/3}.
\end{remark}

\noindent{\bf Methodology:} In the following part, we will explain why we can {\it greatly reduce} the regularity of Sobolev spaces in which the decay estimates and stability result of solutions hold. According to \eqref{nablam}, to establish the global stability of solutions of the system \eqref{equation} in the framework of $H^m(\mt^n)$, the most important point is how to control the term
\begin{equation}\label{ana1}
 \int_{0}^{\infty}\|\nabla u(\tau)\|_{L^\infty(\mathbb{T}^n)}+\|\nabla b(\tau)\|_{L^\infty(\mathbb{T}^n)}\dd \tau.
\end{equation}
Based on classical Sobolev embedding, it holds that
\begin{equation}\label{ana2}
 H^s(\mathbb{T}^n)\hookrightarrow L^\infty(\mathbb{T}^n)\quad {\rm for\,\,\,any}\,\,\,s>\frac n2+1.
\end{equation}
Therefore, by combing \eqref{ana1} and \eqref{ana2}, we know that the core to break through this issue lies in how to find out a suitable Sobolev space $H^s(\mathbb{T}^n)$ for $s>\frac n2+1$ whose decay rate in time is {\it faster than} $(1+t)^{-1}$\,? The lower the value of $s$ is, the better result will be obtained.

However, due to lack of dissipation or damping terms in the equations \eqref{equation}, we can not bound the time integrals of $\|\nabla u\|_{L^\infty(\mathbb{T}^n)}$ and $\|\nabla b\|_{L^\infty(\mathbb{T}^n)}$ directly. The key point to solve the problem is exploiting the hidden dissipation
in the system \eqref{equation}. In the aforementioned works, for the case $\mu=0$ and $\nu=1$, under the assumptions
\eqref{smallcondition2}, the authors in \cite{chen-2022-3dmhd-Diophant} obtained the following dissipative inequality
\begin{align}\notag
   & \frac{\dd}{\dd t}\left(B\left(\|u\|_{H^{r+4}}^2+\|b\|_{H^{r+4}}^2\right)-\sum_{0\leq s\leq{r+3}}\int_{\mathbb{T}^3}\Lambda^s\left(\tb\cdot\nabla u\right)\cdot \Lambda^s b\dd x\right)\label{keyCZZ}\\
   &\,\,\,+\frac{1}{2}\left(B\|\nabla b\|_{H^{r+4}}^2+\frac{1}{2}\|\tilde{b}\cdot\nabla u\|_{{H}^{r+3}}^2\right)\leq 0.
\end{align}
Based on \eqref{keyCZZ}, the authors established the decay rates
\begin{equation}\label{de10}
 \|u(t)\|_{H^{r+4}(\mathbb{T}^3)}+\|b(t)\|_{H^{r+4}(\mathbb{T}^3)}\leq C(1+t)^{-\frac{3}{2}},
\end{equation}
which helped them proving the stability result in $H^{(4r+7)}(\mathbb{T}^3)$. In the subsequent work \cite{zhai-2023-2dmhdstability-Diophant}, for 2D periodic domain $\mathbb{T}^2$ and any $\alpha>0$ and $\beta>0$, the author improved the decay rates \eqref{de10} to
\begin{equation}\label{de100}
 \|u(t)\|_{H^{r+\alpha+3}(\mathbb{T}^2)}+\|b(t)\|_{H^{r+\alpha+3}(\mathbb{T}^2)}\leq C(1+t)^{-(\beta+1)},
\end{equation}
by establishing the different dissipative inequality
\begin{align}\notag
   & \frac{\dd}{\dd t}\left(B\left(\|u\|_{H^{r+\alpha+3}}^2+\|b\|_{H^{r+\alpha+3}}^2\right)-\sum_{0\leq s\leq r+\alpha+2}\int_{\mathbb{T}^2}\Lambda^s\left(\tb\cdot\nabla u\right)\cdot \Lambda^s b\dd x\right)\label{keyzhai}\\
   &\,\,\,+\frac{1}{2}\left(B\|\nabla b\|_{H^{r+\alpha+3}}^2+\|\tilde{b}\cdot\nabla u\|_{{H}^{r+\alpha+2}}^2\right)\leq 0.
\end{align}
This faster decay rates \eqref{de100} enabled the author getting the stability result in Sobolev spaces with lower regularity.
Here, $B>0$ is a constant and the fractional Laplacian operator $\Lambda^s$ is defined via Fourier transform
$$\w{\Lambda^sf}(k)=|k|^{s}\w{f}(k).$$

Be quite different from \cite{chen-2022-3dmhd-Diophant} and \cite{zhai-2023-2dmhdstability-Diophant}, in current work, we find out a {\it new} dissipative-like inequality
\begin{align}
&\frac{\dd}{\dd t}\left(A\left(\|\Lambda^{r+1} u\|_{L^2}^2+\|\Lambda^{r+1} b\|_{L^2}^2\right)-\int_{\mathbb{T}^n}\left(\tb\cdot\nabla u\right)\cdot \Lambda^{2r} b\dd x\right)\label{keyJLX}\\
\leq&  -\frac{c^*}{8M}\left(A\left(\|\Lambda^{r+1} u\|_{L^2}^2+\|\Lambda^{r+1} b\|_{L^2}^2\right)-\int_{\mathbb{T}^n}\left(\tb\cdot\nabla u\right)\cdot \Lambda^{2r} b\dd x\right)+\underbrace{\frac{C}{4M^{1+\frac{m-(r+1)}{r+1}}}\|\Lambda^{m}u\|_{L^2}^2}_{I},\notag
\end{align}
by making effective use of the Diophantine condition and the characteristics of the periodic functions, Fourier analysis and establishing very delicate a priori estimates, where $A=\frac{|\tb|}{2}+1$, $M=A+\frac{c^*t}{8j}$ ($j$ is a suitably large constant), $c^*=\min\{1,c\}$ ($c$ is the constant in \eqref{Diophantine}) and $m\geq 2r+\frac{n}{2}+3$. Thus, we get desired dissipative effect in \eqref{keyJLX}
and the price to pay is the appearance of forcing term $I$. Fortunately, it can be well controlled by constructing the suitable Lyapunov functional
\begin{align*}
	F(t)&\triangleq A\left(\|\Lambda^{\frac{s}{2}+1} u(t)\|_{L^2}^2+\|\Lambda^{\frac{s}{2}+1} b(t)\|_{L^2}^2\right)-\int_{\mathbb{T}^n}\left(\tb\cdot\nabla u(t)\right)\cdot \Lambda^s b(t)\dd x\\
&\geq \left(\|\Lambda^{\frac{s}{2}+1} u\|_{L^2}^2+\|\Lambda^{\frac{s}{2}+1} b\|_{L^2}^2\right),
\end{align*}
with the help of which, we succeed to establish the following decay rate
\begin{align}\label{lowerlower}
		\|u(t)\|_{H^{r+1}(\mathbb{T}^n)}+\|b(t)\|_{H^{r+1}(\mathbb{T}^n)}\leq C(1+t)^{-\frac{m-(r+1)}{2(r+1)}}.
\end{align}
In fact, there are much more details in the process of getting \eqref{lowerlower}, the readers can see Lemma \ref{lem4}
for details. To guarantee that the decay rate in \eqref{lowerlower} is faster than $(1+t)^{-1}$, $m>3(1+r)$ must be true for any dimension $n$, while the corresponding indices in \eqref{de10} and \eqref{de100} are $r+4$ and $r+\alpha+3$ respectively.
It means that we establish the decay estimates in Sobolev spaces with {\it much lower} regularity, which enables us to get the stability result in Sobolev spaces with {\it much lower} regularity.

In addition, to figure out if the decay estimates \eqref{lowerlower} for the nonlinear system is optimal, we first study the linear stability problem. Eventually, by using the method of spectral analysis, for the solutions of linearized system $(U,B)$, we establish the decay rates
\begin{align}\label{lineark0}
\|U\|_{H^{r+1}(\mathbb{T}^n)}+\|B\|_{H^{r+1}(\mathbb{T}^n)}\leq  C (1+t)^{-\frac{m-{(r+1)}}{2(1+r)}},
\end{align}
which show that \eqref{lowerlower} is sharp in the sense that they are in line with \eqref{lineark0}. It is worth noting that according to the Diophantine condition, $r+1>n=\frac{n}{2}+\frac{n}{2}$. Therefore, when $n=2$, $s=r+1$ is the minimum value such that \eqref{ana2} holds. Viewed in this light, for the 2D system, the space $H^{r+1}(\mathbb{T}^2)$ in which the effective decay starts  is {\it optimal} by current research methods.

This paper is organized as follows. In Section 2, plenty of helpful lemmas will be introduced. In Section 3, we concentrate on establishing the linear stability. Section 4 is devoted to the proof of nonlinear stability results.

\section{Preliminaries}\label{2}

\,\,\,\,\,\,\,\,This Section serves as a preparation, which lists several inequalities to be used in the proofs of Theorem \ref{thm} and \ref{thm1}. The first provides a Poincaré-type inequality.

\begin{lemma}\label{2.1}
Assume $\tb\in\mr^n$ satisfies the Diophantine condition \eqref{Diophantine} and $f\in \dot{H}^{s+r+1}$, then for any $s>0$, there holds
	\begin{align}\label{D1}
		\|f\|_{\dot{H}^s}\leq C \|\tilde{b}\cdot\nabla f\|_{\dot{H}^{s+r}}.
	\end{align}
If, in addition $\int_{\mathbb{T}^n}f\dd x=0$, it holds
	\begin{align}\label{D2}
	\|f\|_{L^2}\leq C \|\tilde{b}\cdot\nabla f\|_{\dot{H}^{r}}.
\end{align}
\begin{proof}
According to Plancherel formula, it follows that
\begin{align*}
	\|\tilde{b}\cdot\nabla f\|_{\dot{H}^{s+r}}^2&=\sum_{k\in\mz^n}|k|^{2s+2r}|\tilde{b}\cdot k|^2|\w{f}(k)|^2\\
	&=\sum_{k\in\mz^n\backslash\{0\}}|k|^{2s+2r}|\tilde{b}\cdot k|^2|\w{f}(k)|^2\\
	&\geq c\sum_{k\in\mz^n\backslash\{0\}}|k|^{2s+2r}|k|^{-2r}|\w{f}(k)|^2\\
	&=c\sum_{k\in\mz^n\backslash\{0\}}|k|^{2s}|\w{f}(k)|^2.
\end{align*}
If $s>0$, one has
\begin{align*}
	\sum_{k\in\mz^n\backslash\{0\}}|k|^{2s}|\w{f}(k)|^2= \sum_{k\in\mz^n}|k|^{2s}|\w{f}(k)|^2=\|f\|_{\dot{H}^s}^2.
\end{align*}
If $s=0$ and $\int_{\mathbb{T}^n}f\dd x=0$, we have
\begin{align*}
	\sum_{k\in\mz^n\backslash\{0\}}|k|^{2s}|\w{f}(k)|^2=\sum_{k\in\mz^n\backslash\{0\}}|\w{f}(k)|^2=\sum_{k\in\mz^n}|\w{f}(k)|^2=\|f\|_{L^2},
\end{align*}
which completes the proof.
\end{proof}
\end{lemma}

 The next lemma provides a commutator estimates that can be found in \cite{kato-1988-CommutatorEstimatesEuler} and \cite{kenig}.

\begin{lemma}\label{jiaohuanzi}
	Let $s>0,1<p<\infty$ and $\frac{1}{p}=\frac{1}{p_1}+\frac{1}{q_1}=\frac{1}{p_2}+\frac{1}{q_2}$.
	\begin{itemize}
		\item For any $f\in W^{1,p_1}\cap W^{s,q_2},g\in L^{p_2}\cap W^{s-1,q_1}$, there exists an absolute constant $C$ such that
		\begin{align*}
		\left\|\Lambda^s(f g)-f \Lambda^s g\right\|_{L^p} \leq C\left(\|\nabla f\|_{L^{p_1}}\left\|\Lambda^{s-1} g\right\|_{L^{q_1}}+\|g\|_{L^{p_2}}\left\|\Lambda^s f\right\|_{L^{q_2}}\right).
		\end{align*}
	\item  If $f\in L^{p_1}\cap W^{s,q_2},g\in L^{p_2}\cap W^{s,q_1}$, there is an absolute constant $C$ such that
	\begin{align*}
	\left\|\Lambda^s(f g)\right\|_{L^p} \leq C\left(\|f\|_{L^{p_1}}\left\|\Lambda^s g\right\|_{L^{q_1}}+\|g\|_{L^{p_2}}\left\|\Lambda^s f\right\|_{L^{q_2}}\right).
	\end{align*}
\end{itemize}
\end{lemma}

Finally we state an interpolation inequality which comes from \cite{brezis-2019-gag}.

\begin{lemma}\label{chazhi1}
	If $\int_{\mathbb{T}^n}f(x,t)\dd x=0$, there holds that
	\begin{align*}
		\|\Lambda^rf\|_{L^q}\leq 	\|\Lambda^{s_1}f\|^{\theta}_{L^{p_1}}	\|\Lambda^{s_2}f\|^{1-\theta}_{L^{p_2}},
	\end{align*}
where the real numbers $0 \leq s_1 \leq s_2,~ r \geq 0$, $1 \leq p_1, p_2, q \leq \infty$, $\left(s_1, p_1\right) \neq\left(s_2, p_2\right)$ and $\theta \in(0,1)$ such that
\begin{align*}
	& r<s:=\theta s_1+(1-\theta) s_2,~ \frac{1}{q}=\left(\frac{\theta}{p_1}+\frac{1-\theta}{p_2}\right)-\frac{s-r}{n}.
\end{align*}
\end{lemma}

\section{Linear stability}\label{2/3}
\subsection{Integral presentation}\label{3/1}

\,\,\,\,\,\,\,\,As mentioned earlier, to establish the integral presentation of system \eqref{equation}, we need, to find out its hidden wave structure. First of all, we deal with the case $\mu=0$, $\nu=1$.
Through differentiating the velocity equations $\eqref{equation}_1$ with respect to $t$ and applying the Helmholltz-Leray projection $\mathbb{P}=\mathrm{Id}-\nabla{{\Delta}^{-1}}\nabla\cdot$ to it, taking several substitutions and regrouping linear and nonlinear terms, it follows that
\begin{align}\label{uwave}
	\partial_{tt}u&=\partial_t(\tilde{b}\cdot\nabla b)+\partial_t \mathbb{P}\left(b\cdot\nabla b-u\cdot \nabla u\right)\notag\\
	&=\tilde{b}\cdot\nabla \partial_tb+\partial_t \mathbb{P}\left(b\cdot\nabla b-u\cdot \nabla u\right)\notag\\
	&=\tilde{b}\cdot\nabla\left(\Delta b-u \cdot \nabla b+\tb \cdot \nabla u+b \cdot \nabla u\right)+\partial_t \mathbb{P}\left(b\cdot\nabla b-u\cdot \nabla u\right)\\
	&=\Delta\left(\partial_t u-\mathbb{P}\left(b \cdot \nabla b-u\cdot \nabla u\right)\right)+(\tilde{b}\cdot\nabla)^2u+\partial_t \mathbb{P}\left(b\cdot\nabla b-u\cdot \nabla u\right)\notag\\
	&~~~~~~+\tilde{b}\cdot\nabla(b \cdot \nabla u-u \cdot \nabla b)\notag\\
	&=\partial_t\Delta u+(\tilde{b}\cdot\nabla)^2u+\tilde{b}\cdot\nabla(b \cdot \nabla u-u \cdot \nabla b)+(\partial_t-\Delta)\mathbb{P}\left(b\cdot\nabla b-u\cdot \nabla u\right),\notag
\end{align}
Using a similar process as \eqref{uwave}, we also have
\begin{align}
	\partial_{tt}b=\partial_t\Delta b+(\tilde{b}\cdot\nabla)^2b+\partial_t(b\cdot\nabla u-u\cdot\nabla b)+(\tilde{b}\cdot\nabla)\mathbb{P}(b\cdot\nabla b-u\cdot\nabla u).\label{uwave1}
\end{align}
Combining \eqref{uwave} and \eqref{uwave1} leads to
\begin{align}\label{Requation}
	\left\{\begin{array}{l}
\partial_{tt}u=\partial_t\Delta u+(\tilde{b}\cdot\nabla)^2u+F_1,\\
	\partial_{tt}b=\partial_t\Delta b+(\tilde{b}\cdot\nabla)^2b+F_2 .\\
	\nabla\cdot u=\nabla \cdot b=0, \\
	u(x, 0)=u_0(x), b(x, 0)=b_0(x),
\end{array}\right.
\end{align}
where $F_1$ and $F_2$ are nonlinear terms:
\begin{align}
	F_1&=\tilde{b}\cdot\nabla(b \cdot \nabla u-u \cdot \nabla b)+(\partial_t-\Delta)\mathbb{P}\left(b\cdot\nabla b-u\cdot \nabla u\right),\label{case1}\\
	F_2&=\partial_t(b\cdot\nabla u-u\cdot\nabla b)+(\tilde{b}\cdot\nabla)\mathbb{P}(b\cdot\nabla b-u\cdot\nabla u).\label{case12}
\end{align}

As for the case  $\mu=1$, $\nu=0$, the form of wave equations are same as \eqref{Requation}, the only difference lies in that $F_1$ and $F_2$ should be replaced by $G_1$ and $G_2$ respectively, where
\begin{align}
	G_1&=\tilde{b}\cdot\nabla\mathbb{P}(b \cdot \nabla b-u \cdot \nabla u)+(\partial_t-\Delta)\left(b\cdot\nabla u-u\cdot \nabla b\right),\label{case2}\\
	G_2&=\partial_t\mathbb{P}(b\cdot\nabla b-u\cdot\nabla u)+(\tilde{b}\cdot\nabla)(b\cdot\nabla b-u\cdot\nabla u).\label{case21}
\end{align}

For both cases $\mu=0$, $\nu=1$ and $\mu=1$, $\nu=0$, the corresponding linearized system is given by
\begin{align}\label{linear}
	\left\{\begin{array}{l}
		\partial_{tt}U=\partial_t\Delta U+(\tilde{b}\cdot\nabla)^2U,\\
		\partial_{tt}B=\partial_t\Delta B+(\tilde{b}\cdot\nabla)^2B.\\
		\nabla\cdot U=\nabla \cdot B=0, \\
		U(x, 0)=U_0(x), B(x, 0)=B_0(x).
	\end{array}\right.
\end{align}

To better understand the linearized system \eqref{linear}, it is necessary to thoroughly study the following inhomogeneous equation in ${\mathbb{T}^n}$,
\begin{align}\label{linear1}
	\partial_{tt}\phi(x,t)=\partial_t\Delta \phi(x,t)+(\tilde{b}\cdot\nabla)^2\phi(x,t)+f(x,t),
\end{align}
with
\begin{align}\label{linear2}
	\phi(x,0)=\phi_0(x),\quad \partial_t\phi(x,0)=\phi_1(x).
\end{align}

The first step is converting \eqref{linear1}-\eqref{linear2} into an integral form via the method of spectral analysis. To this end, for any $t>0$ and $k\in\mathbb{Z}^n$,
we define following two operators $L_1(t)$ and $L_2(t)$ as
\begin{align}
	\w{L_1(t)\phi}(k,t)=\frac{1}{2}\left(\e^{(-\frac{|k|^2}{2}+\sigma)t}+\e^{(-\frac{|k|^2}{2}-\sigma)t}\right)\w{\phi}(k,t),\label{3.0}
\end{align}
and
\begin{align}
	\w{L_2(t)\phi}(k,t)=\frac{1}{2\sigma}\left(\e^{(-\frac{|k|^2}{2}+\sigma)t}-\e^{(-\frac{|k|^2}{2}-\sigma)t}\right)\w{\phi}(k,t),\label{3.01}
\end{align}
where $\sigma=\sqrt{\frac{|k|^4}{4}-(\tilde{b}\cdot k)^2}$ and $k\in{\mathbb{Z}^n}$.

\begin{lemma}\label{lem1}
Suppose $\phi_0(x)\in H^2(\mathbb{T}^n)$, $\phi_1(x)\in L^2(\mt^n)$ and $f(x,t)\in L^1_{\rm loc}(\mathbb{R}^+;L^2(\mt^n))$, then any solution of equation \eqref{linear1}-\eqref{linear2} is given by
\begin{align}\label{3.1}
	\phi(x,t)=L_1(t)\phi_0(x)+L_2(t)\left(\phi_1(x)-\frac{1}{2}\Delta\phi_0(x)\right)+\int_0^t L_2(t-\tau)f(x,\tau)\dd \tau.
\end{align}
\end{lemma}
\begin{proof}
Taking the Fourier transform on both sides of \eqref{linear1}, we have
   	\begin{align*}
   		\partial_{tt}\widehat{\phi}(k,t)=-\partial_t|k|^2 \widehat{\phi}(k,t)-(\tilde{b}\cdot k)^2\widehat{\phi}(k,t)+\widehat{f}(k,t),\quad k\in{\mathbb{Z}^n},
   	\end{align*}
which implies, after some basic calculations, that
   \begin{align*}
   	\left(\left(\partial_t+\frac{|k|^2}{2}\right)^2-\frac{|k|^4}{4}+(\tilde{b}\cdot k)^2\right)\widehat{\phi}(k,t)=\widehat{f}(k,t),
   \end{align*}
namely
\begin{align}\label{3.2}
	\left(\partial_t+\frac{|k|^2}{2}-\sigma\right)\left(\partial_t+\frac{|k|^2}{2}+\sigma\right)\widehat{\phi}(k,t)(k,t)=\widehat{f}(k,t).
\end{align}

To solve \eqref{3.2}, we assume
\begin{align}\label{3.3}
		\left(\partial_t+\frac{|k|^2}{2}+\sigma\right)\widehat{\phi}(k,t)=\Phi^+(k,t),~	\left(\partial_t+\frac{|k|^2}{2}-\sigma\right)\widehat{\phi}(k,t)=\Phi^-(k,t),
\end{align}
and therefore it satisfies
\begin{align*}
	\left(\partial_t+\frac{|k|^2}{2}-\sigma\right)\Phi^+(k,t)=\widehat{f}(k,t),~	\left(\partial_t+\frac{|k|^2}{2}+\sigma\right)\Phi^-(k,t)=\widehat{f}(k,t).
\end{align*}
By Duhamel’s principle, it follows that
\begin{align}\label{3.4}
	\Phi^+(k,t)=\e^{(-\frac{|k|^2}{2}+\sigma)t}\Phi^+(k,0)+\int_0^t\e^{(-\frac{|k|^2}{2}+\sigma)(t-\tau)}\widehat{f}(k,\tau) \dd \tau,\\
	\Phi^-(k,t)=\e^{(-\frac{|k|^2}{2}-\sigma)t}\Phi^-(k,0)+\int_0^t\e^{(-\frac{|k|^2}{2}-\sigma)(t-\tau)}\widehat{f}(k,\tau) \dd \tau,\label{3.5}
\end{align}
which implies, by taking $t=0$ in \eqref{3.3} and then substituting into \eqref{3.4} and \eqref{3.5}, that
\begin{align}
	&\w{\phi}(k,t)=\frac{\Phi^+(k,t)-\Phi^-(k,t)}{2\sigma}\notag\\
	&=\frac{1}{2\sigma}\left[\left(\e^{(-\frac{|k|^2}{2}+\sigma)t}-\e^{(-\frac{|k|^2}{2}-\sigma)t}\right) \w{\phi_1}(k)+\frac{|k|^2}{2}\w{\phi}_0(k)\left(\e^{(-\frac{|k|^2}{2}+\sigma)t}-\e^{(-\frac{|k|^2}{2}-\sigma)t}\right)\right]\label{phi1}\\ &~~~+\frac{1}{2}\w{\phi_0}(k)\left(\e^{(-\frac{|k|^2}{2}+\sigma)t}+\e^{(-\frac{|k|^2}{2}-\sigma)t}\right)+\frac{1}{2\sigma}\int_0^t\left(\e^{(-\frac{|k|^2}{2}+\sigma)(t-\tau)}-\e^{(-\frac{|k|^2}{2}-\sigma)(t-\tau)}\right)\w{f}(k,\tau)\dd \tau.\notag
\end{align}

Recalling the definitions of $L_1(t)$ and $L_2(t)$ in \eqref{3.0} and \eqref{3.01}, one can update \eqref{phi1} as
\begin{align}\label{phi2}
		\w{\phi}(k,t)=\w{L_1(t)\phi_0}(k)+\left(\w{L_2(t)\phi_1}(k)+\frac{|k|^2}{2}\w{L_2(t)\phi_0}(k)\right)+\int_0^t \w{L_2(t-\tau)f}(k,\tau)\dd \tau.
\end{align}
Finally, we can obtain \eqref{3.1} by taking the Fourier inverse transform on both sides of \eqref{phi2}.
\end{proof}

In order to simplify the presentation, we introduce the following two numbers that will be used frequently
\begin{align*}
	\lambda_1=(-\frac{|k|^2}{2}+\sigma)=-\frac{|k|^2}{2}+\sqrt{\frac{|k|^4}{4}-(\tilde{b}\cdot k)^2},\\
	\lambda_2=(-\frac{|k|^2}{2}-\sigma)=-\frac{|k|^2}{2}-\sqrt{\frac{|k|^4}{4}-(\tilde{b}\cdot k)^2},
\end{align*}
and then \eqref{3.0}-\eqref{3.01} can be simplified as
\begin{align}\label{3.8}
	\w{L_1(t)\phi}(k,t)=\frac{1}{2}\left(\e^{(-\frac{|k|^2}{2}+\sigma)t}+\e^{(-\frac{|k|^2}{2}-\sigma)t}\right)\w{\phi}(k,t)=\frac{1}{2}\left(\e^{\lambda_1t}+\e^{\lambda_2t}\right)\w{\phi}(k,t),\\
	\w{L_2(t)\phi}(k,t)=\frac{1}{2\sigma}\left(\e^{(-\frac{|k|^2}{2}+\sigma)t}-\e^{(-\frac{|k|^2}{2}-\sigma)t}\right)\w{\phi}(k,t)=\frac{1}{2\sigma}\left(\e^{\lambda_1t}-\e^{\lambda_2t}\right)\w{\phi}(k,t).\label{3.9}
\end{align}

\subsection{Spectral analysis}\label{3}

\,\,\,\,\,\,\,\,To establish the linear stability, large time behavior and nonlinear stability of $u$ and $b$, we need estimate the operators $L_1(t)$ and $L_2(t)$ precisely.
Because the behavior of them depends on the frequency $k$, it is natural to divide the frequency space $\mz^n\backslash\{0\}$ into following three subdomains according to the range of $k$, that is
	\begin{align}
	S_1&=\{k\in\mathbb{Z}^n\backslash\{0\}: 1-4\frac{(\tilde{b}\cdot k)^2}{|k|^4}\leq 0\},\label{s1}\\
	S_2&=\{k\in\mathbb{Z}^n\backslash\{0\}:  0\leq1-4\frac{(\tilde{b}\cdot k)^2}{|k|^4}\leq \frac{1}{4}\},\label{s2}\\
	S_3&=\{k\in\mathbb{Z}^n\backslash\{0\}: 1-4\frac{(\tilde{b}\cdot k)^2}{|k|^4}\geq\frac{1}{4}\},\label{s3}
\end{align}
where $S_1$, $S_2$ and $S_3$ represent {\it low frequency space}, {\it medium frequency space} and {\it high frequency space} respectively. Thanks to \eqref{s1}-\eqref{s3}, we can summarize the precise estimates of $L_1(t)$ and $L_2(t)$ in the following lemma.
\begin{lemma}\label{lem2}
 Let $f$ be a function defined on $\mathbb{T}^n$, then for the absolute constant $C$, $L_1(t)$ and $L_2(t)$ satisfy the following estimates:
\begin{itemize}
\item For $k\in S_1$
	\begin{align*}
		|\w{L_1f}(k,t)|\leq \e^{-\frac{|k|^2}{2}t}|\w{f}|, ~
		|\w{L_2f}(k,t)|\leq C \frac{1}{|k|^2}\e^{-\frac{|k|^2}{4}t}|\w{f}|.
	\end{align*}
\item For $k\in S_2$
\begin{align*}
	|\w{L_1f}(k,t)|\leq C \e^{-\frac{|k|^2}{4}t}|\w{f}|,~
	|\w{L_2f}(k,t)|\leq C\frac{1}{|k|^2}\e^{-\frac{|k|^2}{8}t}|\w{f}|.
\end{align*}
\item For $k\in S_3$
\begin{align*}
	|\w{L_1f}(k,t)|\leq \frac{1}{2}\left(\e^{-\frac{(\tilde{b}\cdot k)^2}{|k|^2} t}+\e^{-\frac{3|k|^2}{4}t}\right)|\w{f}|,~
	|\w{L_2f}(k,t)|\leq \frac{\e^{-\frac{(\tilde{b}\cdot k)^2}{|k|^2} t}+\e^{-\frac{3|k|^2}{4}t}}{|k|^2}|\w{f}|.
\end{align*}
\end{itemize}
\end{lemma}

\begin{proof}
{\bf When $k\in S_1$}, by utilizing the Euler formula
\begin{align}\label{euler}
	\e^{\ii x}=\cos x+\ii \sin x,
\end{align}
and noticing that $\left|\cos x\right| \leq 1$ for any $x \in \mr$, we can obtain
\begin{align*}
	|\w{L_1f}(k,t)|=\left|\frac{\e^{\lambda_1t}+\e^{\lambda_2t}}{2}\w{f}\right|&=\left|\frac{\e^{(-\frac{|k|^2}{2}+\sigma) t}+\e^{(-\frac{|k|^2}{2}-\sigma)t}}{2}\right||\w{f}|\\
	&=\e^{-\frac{|k|^2}{2}t}\left|\frac{\e^{\ii \frac{|k|^2}{2}t\sqrt{4\frac{(\tilde{b}\cdot k)^2}{|k|^4}-1}}+\e^{-\ii \frac{|k|^2}{2}t\sqrt{4\frac{(\tilde{b}\cdot k)^2}{|k|^4}-1}}}{2}\right||\w{f}|\\
	&=\e^{-\frac{|k|^2}{2}t}\left|\cos \left(\frac{|k|^2}{2}t\sqrt{4\frac{(\tilde{b}\cdot k)^2}{|k|^4}-1}\right)\right||\w{f}|\\
	&\leq \e^{-\frac{|k|^2}{2}t}|\w{f}|.
\end{align*}
By employing \eqref{euler} again, the inequality $\left| \sin x\right|\leq |x|$ holds for any $x\in \mr$ and
\begin{equation}\label{ye}
  y^n\e^{-y}\leq C, \quad {\rm for\,\,any}\,\, y\in \mathbb{R}^+\,\,{\rm and}\,\, n\in \mathbb{N}^+,
\end{equation}
it follows that

\begin{align*}
	|\w{L_2f}(k,t)|=\left|\frac{\e^{\lambda_1t}-\e^{\lambda_2t}}{\lambda_1-\lambda_2}\w{f}\right|&=\left|\frac{\e^{(-\frac{|k|^2}{2}+\sigma) t}-\e^{(-\frac{|k|^2}{2}-\sigma)t}}{2\sigma}\right||\w{f}|\\
&=\e^{-\frac{|k|^2}{2}t}\left|\frac{\e^{\ii \frac{|k|^2}{2}t\sqrt{4\frac{(\tilde{b}\cdot k)^2}{|k|^4}-1}}-\e^{-\ii \frac{|k|^2}{2}t\sqrt{4\frac{(\tilde{b}\cdot k)^2}{|k|^4}-1}}}{2\ii \frac{|k|^2}{2}\sqrt{4\frac{(\tilde{b}\cdot k)^2}{|k|^4}-1}}\right||\w{f}|\\
&\leq t\e^{-\frac{|k|^2}{2}t}\left|\frac{2\ii\sin \left(\frac{|k|^2}{2}t\sqrt{4\frac{(\tilde{b}\cdot k)^2}{|k|^4}-1}\right)}{2\ii \frac{|k|^2}{2}t\sqrt{4\frac{(\tilde{b}\cdot k)^2}{|k|^4}-1}}\right||\w{f}|\\
&\leq C \frac{1}{|k|^2}\e^{-\frac{|k|^2}{4}t}|\w{f}|.
\end{align*}

{\bf When $k\in S_2$}, $1-4\frac{(\tilde{b}\cdot k)^2}{|k|^4}$ has upper and lower bounds, and so are $\lambda_1$ and $\lambda_2$ because
\begin{align*}
	-\frac{|k|^2}{2}\leq\lambda_1=-\frac{|k|^2}{2}\left(1-\sqrt{1-4\frac{(\tilde{b}\cdot k)^2}{|k|^4}}\right)\leq -\frac{|k|^2}{4},\\
	-\frac{3|k|^2}{4}\leq\lambda_2=-\frac{|k|^2}{2}\left(1+\sqrt{1-4\frac{(\tilde{b}\cdot k)^2}{|k|^4}}\right)\leq-\frac{|k|^2}{2}.
\end{align*}
Therefore, it follows that
\begin{align}
	|\w{L_1f}(k,t)|=\left|\frac{\e^{\lambda_1t}+\e^{\lambda_2t}}{2}\w{f}\right|\leq  \frac{1}{2}\left(\e^{-\frac{|k|^2}{2}t}+\e^{-\frac{|k|^2}{4}t}\right)|\w{f}|\leq C \e^{-\frac{|k|^2}{4}t}|\w{f}|,\label{l11}
\end{align}
and
\begin{align}
	|\w{L_2f}(k,t)|=\left|\frac{\e^{\lambda_1t}-\e^{\lambda_2t}}{\lambda_1-\lambda_2}\w{f}\right|=t\e^{\xi t}|\w{f}| \leq C\frac{1}{|k|^2}\e^{-\frac{|k|^2}{8}t}|\w{f}|,\label{l21}
\end{align}
where we used the mean value theorem and \eqref{ye} in \eqref{l21} and $\xi\in (\lambda_2,\lambda_1)$.

{\bf When $k\in S_3$}, $\lambda_1$ may approaches zero and corresponding estimates will not depend on $|k|^2$. To solve this problem, we rewrite
$\lambda_1$ as follows
\begin{align*}
		\lambda_1=-\frac{|k|^2}{2}\left(1-\sqrt{1-4\frac{(\tilde{b}\cdot k)^2}{|k|^4}}\right)=\dfrac{-2\frac{(\tilde{b}\cdot k)^2}{|k|^2}}{1+\sqrt{1-4\frac{(\tilde{b}\cdot k)^2}{|k|^4}}}.
	\end{align*}
Based on the fact that $S_3$ satisfies $\frac{1}{4}\leq 1-4\frac{(\tilde{b}\cdot k)^2}{|k|^4}\leq 1$, we have
\begin{align*}
 &-\frac{4(\tilde{b}\cdot k)^2}{3|k|^2} \leq	\lambda_1\leq -\frac{(\tilde{b}\cdot k)^2}{|k|^2} ,\\
-|k|^2\leq \lambda_2&=-\frac{|k|^2}{2}\left(1+\sqrt{1-4\frac{(\tilde{b}\cdot k)^2}{|k|^4}}\right)\leq-\frac{3|k|^2}{4},
\end{align*}
and
\begin{align*}
	2\sigma=\lambda_1-\lambda_2=2\sqrt{\frac{|k|^4}{4}-(\tilde{b}\cdot k)^2}=|k|^2\sqrt{1-4\frac{(\tilde{b}\cdot k)^2}{|k|^4}}\geq\frac{|k|^2}{2}.
\end{align*}
One has
\begin{align*}
	|\w{L_1f}(k,t)|=\left|\frac{\e^{\lambda_1t}+\e^{\lambda_2t}}{2}\w{f}\right|\leq  \frac{1}{2}\left(\e^{-\frac{(\tilde{b}\cdot k)^2}{|k|^2} t}+\e^{-\frac{3|k|^2}{4}t}\right)|\w{f}|,
\end{align*}
and
\begin{align*}
	|\w{L_2f}(k,t)|=\left|\frac{\e^{\lambda_1t}-\e^{\lambda_2t}}{2\sigma}\w{f}\right|\leq\frac{\e^{-\frac{(\tilde{b}\cdot k)^2}{|k|^2} t}+\e^{-\frac{3|k|^2}{4}t}}{|k|^2}|\w{f}|.
\end{align*}
We have completed the proof of this Lemma.
\end{proof}

\subsection{Proofs for the linear stability results}

\,\,\,\,\,\,\,\,In what follows, we get to establish the linear stability results by Lemma \ref{lem1} and \ref{lem2}.
As a preparation, we first give the integral representation of $U$ and $B$. For the convenience of reader,
we recall the linearized system \eqref{linear},
\begin{align}\label{linearrepeat}
	\left\{\begin{array}{l}
		\partial_{tt}U=\partial_t\Delta U+(\tilde{b}\cdot\nabla)^2U,\\
		\partial_{tt}B=\partial_t\Delta B+(\tilde{b}\cdot\nabla)^2B,\\
        \nabla\cdot U=\nabla \cdot B=0, \\
		U(x, 0)=U_0(x), B(x, 0)=B_0(x),
	\end{array}\right.
\end{align}
and the linearized form of \eqref{equation},
\begin{align}\label{linear21}
	\left\{\begin{array}{l}
		\partial_t U-\mu\Delta U=\tilde{b} \cdot \nabla B, \\
		\partial_t B-\nu\Delta B=\tilde{b} \cdot \nabla U.
	\end{array}\right.
\end{align}

For the case  $\mu=0$, $\nu=1$, according to Lemma \ref{lem1} and \eqref{linear21}, the integral representations of $U$ and $B$ are given by
\begin{align}
		U(x,t)&=L_1(t)U_0(x)+L_2(t)\left(\tilde{b} \cdot \nabla B_0(x)-\frac{1}{2}\Delta U_0(x)\right),\label{new1}\\
		B(x,t)&=L_1(t)B_0(x)+L_2(t)\left(\tilde{b} \cdot \nabla U_0(x)+\frac{1}{2}\Delta B_0(x)\right).\label{new2}
\end{align}
For the case $\mu=1$, $\nu=0$, similarly, we have
\begin{align}
		U(x,t)&=L_1(t)U_0(x)+L_2(t)\left(\tilde{b} \cdot \nabla B_0(x)+\frac{1}{2}\Delta U_0(x)\right),\label{new3}\\
		B(x,t)&=L_1(t)B_0(x)+L_2(t)\left(\tilde{b} \cdot \nabla U_0(x)-\frac{1}{2}\Delta B_0(x)\right).\label{new4}
\end{align}

Now, we turn to the proof of Theorem \ref{thm1}.

\begin{proof}[Proof of Theorem \ref{thm1}]
It is clear that the proofs for case $\mu=0$, $\nu=1$ and case $\mu=1$, $\nu=0$ are same except for the positive or negative of $\Delta U_0$ and $\Delta B_0$,
to avoid duplication, we only provide the first one. Due to \eqref{mean-zero}, for any $t\geq 0$, there holds $\int_{\mathbb{T}^n}U(x,t)\dd x=\int_{\mathbb{T}^n}B(x,t)\dd x=0$ and therefore
\begin{equation}\label{mean-zero2}
 \w{U}(0,t)=0,\quad  \w{B}(0,t)=0.
\end{equation}
According to \eqref{new1}, \eqref{mean-zero2} and Plancherel’s theorem, it follows that
\begin{align}\nonumber
	\|U(x,t)\|^2_{L^2}&=\sum_{k\neq 0}|\w{U}(k,t)|^2\\\nonumber
	&	\leq 3\sum_{k\neq 0}|\w{L_1(t)U}_0(k)|^2+\frac32\sum_{k\neq 0}\left|\w{L_2(t)\Delta U}_0(k)\right|^2+3\sum_{k\neq 0}\left|\w{L_2(t)\tilde{b} \cdot \nabla B}_0(x)\right|^2\\
	&\triangleq I_1+I_2+I_3.\label{Usplit}
\end{align}
Utilizing Lemma \ref{lem2} and the Diophantine condition \eqref{Diophantine}, we have
\begin{align}
	I_1=3\sum_{k\neq 0}|\w{L_1(t)U}_0(k)|^2&\leq C\sum_{k\in S_1\cup S_2}\e^{-\frac{|k|^2}{2}t}|\w{U}_0(k)|^2+C\sum_{k\in S_3}\left(\e^{-\frac{2(\tilde{b}\cdot k)^2}{|k|^2} t}+\e^{-\frac{3|k|^2}{4}t}\right)|\w{U}_0(k)|^2\notag\\
	&\leq C \sum_{k\neq 0}\e^{-\frac{|k|^2}{2}t}|\w{U}_0(k)|^2+C\sum_{k\in S_3}\e^{-\frac{2(\tilde{b}\cdot k)^2}{|k|^2} t}|\w{U}_0(k)|^2\label{I11}\\
	&\leq C\e^{-\frac{t}{2}}\|U_0(x)\|_{L^2}^2+C\sum_{k\in S_3}\e^{-\frac{2c^2}{|k|^{2+2r}} t}|\w{U}_0(k)|^2.\notag
\end{align}
With the help of \eqref{ye}, the second term can be estimated as
\begin{align}
	&\sum_{k\in S_3}\e^{-\frac{2c^2}{|k|^{2+2r}} t}|\w{U}_0(k)|^2\notag\\
	=&\sum_{k\in S_3}\e^{-\frac{2c^2}{|k|^{2+2r}} t}\left(\frac{t}{|k|^{2+2r}}\right)^{\frac{2s}{2+2r}}t^{-\frac{2s}{2+2r}}|k|^{2s}|\w{U}_0(k)|^2\label{I12}\\
	\leq&  t^{-\frac{2s}{2+2r}} \|U_0(x)\|_{\dot{H}^{s}}^2\sup_{k\in S_3}\e^{-\frac{2c^2}{|k|^{2+2r}} t}\left(\frac{t}{|k|^{2+2r}}\right)^{\frac{2s}{2+2r}}\notag\\
	\leq& Ct^{-\frac{2s}{2+2r}} \|U_0(x)\|_{\dot{H}^{s}}^2,\notag
\end{align}	
or
\begin{align}
	\sum_{k\in S_3}\e^{-\frac{2c^2}{|k|^{2+2r}} t}|\w{U}_0(k)|^2\leq C\|U_0(x)\|_{L^2}^2.\label{I13}
\end{align}
Then by combining \eqref{I11}-\eqref{I13}, it yields that
\begin{align}\label{I_1}
	I_1\leq C\e^{-\frac{t}{2}}\|U_0(x)\|_{L^2}^2+C (1+t)^{-\frac{2s}{2+2r}}\|U_0(x)\|_{H^s}^2.
\end{align}

Employing Lemma \ref{lem2} once more, we can derive
\begin{align*}
	I_2&=\frac32\sum_{k\neq 0}\left|\w{L_2(t)\Delta U}_0(k)\right|^2\\
	&\leq C\sum_{k\in S_1\cup S_2}\frac{1}{|k|^4}\e^{-\frac{|k|^2}{4}t}|k|^4|\w{U}_0(k)|^2+C\sum_{k\in S_3}\left(\frac{1}{|k|^4}\e^{-\frac{2(\tilde{b}\cdot k)^2}{|k|^2} t}+\frac{1}{|k|^4}\e^{-\frac{3|k|^2}{2}t}\right)|k|^4|\w{U}_0(k)|^2\\
	&\leq C\sum_{k\neq0}\e^{-\frac{|k|^2}{4}t}|\w{U}_0(k)|^2+C\sum_{k\in S_3}\e^{-\frac{2(\tilde{b}\cdot k)^2}{|k|^2} t}|\w{U}_0(k)|^2,
\end{align*}
which implies, after using similar method as estimating $I_1$, that
\begin{align}\label{I_2}
	I_2\leq C\e^{-\frac{t}{4}}\|U_0(x)\|_{L^2}^2+C(1+t)^{-\frac{2s}{2+2r}}\|U_0(x)\|_{H^s}^2.
\end{align}

Noticing that the vector $\tb$ is given, it holds that
\begin{align*}
	I_3&=3\sum_{k\neq 0}\left|\w{L_2(t)\tilde{b} \cdot \nabla B}_0(x)\right|^2\\
	&\leq C\sum_{k\in S_1\cup S_2}\frac{1}{|k|^4}\e^{-\frac{|k|^2}{2}t}|\tilde{b}\cdot k\w{B}_0(k)|^2+C\sum_{k\in S_3}\left(\frac{1}{|k|^4}\e^{-\frac{2(\tilde{b}\cdot k)^2}{|k|^2} t}+\frac{1}{|k|^4}\e^{-\frac{3|k|^2}{2}t}\right)|\tilde{b}\cdot k\w{B}_0(k)|^2\\
	&\leq C\sum_{k\neq 0}\e^{-\frac{|k|^2}{2}t}\frac{1}{|k|^2}|\w{B}_0(k)|^2+C\sum_{k\in S_3}\e^{-\frac{2(\tilde{b}\cdot k)^2}{|k|^2} t}\frac{1}{|k|^2}|\w{B}_0(k)|^2,
\end{align*}
which yields, after utilizing similar approach as estimating $I_1$, that
\begin{align}\label{I_3}
	I_3\leq C\e^{-\frac{t}{2}}\|B_0(x)\|_{L^2}^2+C(1+t)^{-\frac{2s+2}{2+2r}}\|B_0(x)\|_{H^s}^2.
\end{align}

Next, by substituting \eqref{I_1}-\eqref{I_3} into \eqref{Usplit}, one can get
\begin{align}
	\|U\|_{L^2}\leq &C\e^{-\frac{t}{4}}\left(\|U_0\|_{L^2}+\|B_0\|_{L^2}\right)+C (1+t)^{-\frac{s}{2+2r}}\|U_0\|_{H^s}+C(1+t)^{-\frac{s+1}{2+2r}}\|B_0\|_{H^s}.\label{ul2}
  \end{align}
Because $\tb$ is given, from the linearized system \eqref{linearrepeat}, it is clear that the equations of $U(x,t)$ and $B(x,t)$ have similar structure and thus we can analogously conclude
\begin{align}
	\|B\|_{L^2}\leq& C\e^{-\frac{t}{4}}\left(\|U_0\|_{L^2}+\|B_0\|_{L^2}\right)+C (1+t)^{-\frac{s}{2+2r}}\|B_0\|_{H^s}+C(1+t)^{-\frac{s+1}{2+2r}}\|U_0\|_{H^s}.\label{bl2}
  \end{align}
Summing up \eqref{ul2} and \eqref{bl2} imply \eqref{lineark} with $\alpha=0$.

On account of the linearity of \eqref{linearrepeat}, the system satisfied by $\Lambda^\alpha U(x,t)$ and $\Lambda^\alpha B(x,t)$ is same as $U(x,t)$
and $B(x,t)$. Therefore, by repeating above process, for any $0\leq \alpha\leq s$, there holds
\begin{align}\label{linearalpha}
	\|\Lambda^\alpha U(x,t)\|_{L^2}+\|\Lambda^\alpha B(x,t)\|_{L^2}\leq  C (1+t)^{-\frac{s-\alpha}{2+2r}}\left(\|U_0(x)\|_{H^s}+\|B_0(x)\|_{H^s}\right),
\end{align}
which yields \eqref{lineark} after adding \eqref{ul2} and \eqref{bl2}. This completes the proof.
\end{proof}

\section{Nonlinear stability}\label{4}
\subsection{A priori estimates}\label{4/1}

\,\,\,\,\,\,\,\,In order to unearth the effective decay estimates, for the case $\mu=0$, $\nu=1$, we construct a suitable Lyapunov functional,
\begin{align}
	F(t)\triangleq A\left(\|\Lambda^{\frac{s}{2}+1} u(t)\|_{L^2}^2+\|\Lambda^{\frac{s}{2}+1} b(t)\|_{L^2}^2\right)-\int_{\mathbb{T}^n}\left(\tb\cdot\nabla u(t)\right)\cdot \Lambda^s b(t)\dd x,\label{lya}
\end{align}
where $A$ is a suitably selected content. The attention will be focused on the evolution of $F(t)$. Following is about the first term on the right hand of \eqref{lya}.

\begin{lemma}\label{hign energy}
	For any $\lambda\in[0,m]$ and $t\in[0,T]$, it holds that
	\begin{align}
	&\frac{1}{2}	\frac{\dd}{\dd t}\left(\|\Lambda^\lambda  u\|_{L^2}^2+\|\Lambda^\lambda  b\|_{L^2}^2\right)+\mu\|\Lambda^{\lambda+1} u\|_{L^2}^2+\nu\|\Lambda^{\lambda+1} b\|_{L^2}^2\notag\\
\leq& C \left(\|\nabla u\|_{L^\infty}+\|\nabla b\|_{L^\infty}\right)\left(\|\Lambda^\lambda  u\|_{L^2}^2+\|\Lambda^\lambda  b\|_{L^2}^2\right).\label{energy}
\end{align}
\end{lemma}
\begin{proof}
Applyling the operator $\Lambda^\lambda$ on both sides of \eqref{equation}$_1$ and \eqref{equation}$_2$, taking the inner product of resultants with $\Lambda^\lambda u$ and $\Lambda^\lambda b$ respectively and adding them up, we have
\begin{align*}
\frac{1}{2}	\frac{\dd}{\dd t}\left(\|\Lambda^\lambda  u\|_{L^2}^2+\|\Lambda^\lambda  b\|_{L^2}^2\right)+\mu\|\Lambda^{\lambda+1} u\|_{L^2}^2+\nu\|\Lambda^{\lambda+1} b\|_{L^2}^2=\sum_{j=1}^{6}I_j,
\end{align*}
where
\begin{align*}
	I_1=&- \int_{\mathbb{T}^n} \Lambda^\lambda  (u \cdot \nabla u) \cdot \Lambda^\lambda u \mathrm{~d} x,
	~I_2= \int_{\mathbb{T}^n}\Lambda^\lambda (b \cdot \nabla b) \cdot \Lambda^\lambda  u \mathrm{~d} x, \\
	I_3=& - \int_{\mathbb{T}^n}\Lambda^\lambda  (u \cdot \nabla b) \cdot \Lambda^\lambda  b\mathrm{~d} x,
	~I_4= \int_{\mathbb{T}^n} \Lambda^\lambda  (b \cdot \nabla u) \cdot \Lambda^\lambda  b\mathrm{~d} x,\\
	I_5=& \int_{\mathbb{T}^n} (\tilde{b}\cdot\nabla \Lambda^\lambda  b) \cdot \Lambda^\lambda  u \mathrm{~d} x, ~I_6= \int_{\mathbb{T}^n} (\tilde{b}\cdot\nabla \Lambda^\lambda  u)\cdot \Lambda^\lambda  b\mathrm{~d} x.
\end{align*}
By some basic calculations, it clearly holds
\begin{align*}
	I_5+I_6=\int_{\mathbb{T}^n} (\tilde{b}\cdot\nabla \Lambda^\lambda  b) \cdot \Lambda^\lambda  u \mathrm{~d} x+ \int_{\mathbb{T}^n} (\tilde{b}\cdot\nabla \Lambda^\lambda  u)\cdot \Lambda^\lambda  b\mathrm{~d} x=0.
\end{align*}

Making use of Lemma \ref{jiaohuanzi} and \eqref{equation}$_3$, it follows that
\begin{align*}
	I_1=- \int_{\mathbb{T}^n} \Lambda^\lambda  (u \cdot \nabla u) \cdot \Lambda^\lambda u \mathrm{~d} x&=- \int_{\mathbb{T}^n} [\Lambda^\lambda  (u \cdot \nabla u) -u \cdot  \nabla \Lambda^\lambda  u]\cdot \Lambda^\lambda u \mathrm{~d} x\\
  &\leq \|\Lambda^\lambda  u\|^2_{L^2}\|\nabla u\|_{L^\infty},\\
  I_3= -\int_{\mathbb{T}^n}\Lambda^\lambda (u \cdot \nabla b) \cdot \Lambda^\lambda  b \mathrm{~d} x&=- \int_{\mathbb{T}^n} [\Lambda^\lambda  (u \cdot \nabla b) -u \cdot  \nabla \Lambda^\lambda  b]\cdot \Lambda^\lambda b \mathrm{~d} x\\
  &\leq \|\Lambda^\lambda  b\|_{L^2}\left(\|\Lambda^\lambda  u\|_{L^2}\|\nabla b\|_{L^\infty}+\|\Lambda^\lambda  b\|_{L^2}\|\nabla u\|_{L^\infty}\right),
\end{align*}
and
\begin{align*}
	I_2+I_4&=- \int_{\mathbb{T}^n}\Lambda^\lambda (b \cdot \nabla b) \cdot \Lambda^\lambda  u \mathrm{~d} x+ \int_{\mathbb{T}^n} \Lambda^\lambda  (b \cdot \nabla u) \cdot \Lambda^\lambda  b\mathrm{~d} x\\
	&=- \int_{\mathbb{T}^n}[\Lambda^\lambda (b \cdot \nabla b)-b \cdot \Lambda^\lambda \nabla  b] \cdot \Lambda^\lambda  u \mathrm{~d} x+\int_{\mathbb{T}^n} [\Lambda^\lambda  (b \cdot \nabla u) -b \cdot \Lambda^\lambda  \nabla u]\cdot \Lambda^\lambda  b\mathrm{~d} x\\
	&\leq \|\Lambda^\lambda  u\|_{L^2}\|\Lambda^\lambda  b\|_{L^2}\|\nabla u\|_{L^\infty}+\|\Lambda^\lambda  b\|_{L^2}\left(\|\Lambda^\lambda  u\|_{L^2}\|\nabla b\|_{L^\infty}+\|\Lambda^\lambda  b\|_{L^2}\|\nabla u\|_{L^\infty}\right).
\end{align*}
Collecting all estimates above leads to \eqref{energy}, which completes the proof.
\end{proof}

Next work is devoted to the second term on the right hand of \eqref{lya}, which will bring out the dissipation of $u$ and $b$.

\begin{lemma}\label{lem3}
	For any $s\in\mr^+$, the following estimate holds
\begin{align}
	-\frac{\dd}{\dd t}\int_{\mathbb{T}^n}\left(\tb\cdot\nabla u\right)\cdot \Lambda^s b\dd x&\leq \frac{3}{2}\left\|\Lambda^{\frac{s}{2}+2}b\right\|_{L^2}-\frac{1}{2}\left\|\Lambda^{\frac{s}{2}}\left(\tb\cdot\nabla u\right)\right\|_{L^2}\label{key1}\\
	&~~~~+\left(\|\nabla b\|_{L^\infty}+\|\nabla u\|_{L^\infty}\right) \left(\|\Lambda^{\frac{s}{2}+1} b\|_{L^2}^2+\|\Lambda^{\frac{s}{2}+1} u\|_{L^2}^2\right).\notag
\end{align}
\end{lemma}
\begin{proof}
Initially, by some direct calculations, it follows that
	\begin{align}\nonumber
			-\frac{\dd}{\dd t}\int_{\mathbb{T}^n}\left(\tb\cdot\nabla u\right)\cdot \Lambda^s b\dd x&=-\int_{\mathbb{T}^n}\left(\tb\cdot\nabla \partial_t u\right)\cdot \Lambda^s b-\left(\tb\cdot\nabla u\right)\cdot \Lambda^s \partial_t b\dd x\\
			&\triangleq J_1+J_2.\label{4.1}
	\end{align}

To estimate $J_1$, we can use \eqref{equation}$_1$ to update it as
\begin{align*}
	J_1&=-\int_{\mathbb{T}^n}\left(\tb\cdot\nabla \left(\tilde{b} \cdot \nabla b+\mathbb{P}(b \cdot \nabla b-u \cdot \nabla u)\right)\right)\cdot \Lambda^s b\dd x\\
	&\leq \|\Lambda^{\frac{s}{2}+1} b\|_{L^2}\|(\tb\cdot\nabla)\Lambda^{\frac{s}{2}-1}\mathbb{P}(b \cdot \nabla b-u \cdot \nabla u)\|_{L^2}-\int_{\mathbb{T}^n}\left(\tb\cdot\nabla \right)^2b\cdot \Lambda^s b\dd x,
\end{align*}
which implies, after using Lemma \ref{jiaohuanzi} and integration by parts, that
\begin{align}\label{J1}
		J_1&\leq\|\Lambda^{\frac{s}{2}+1} b\|_{L^2}\left(\|\Lambda^{\frac{s}{2}+1} b\|_{L^2}\|\nabla b\|_{L^\infty}+\|\Lambda^{\frac{s}{2}+1} u\|_{L^2}\|\nabla u\|_{L^\infty}\right)-\left\|\Lambda^{\frac{s}{2}}\left(\tb\cdot\nabla b\right)\right\|^2_{L^2}\\
&\leq\|\Lambda^{\frac{s}{2}+1} b\|_{L^2}\left(\|\Lambda^{\frac{s}{2}+1} b\|_{L^2}\|\nabla b\|_{L^\infty}+\|\Lambda^{\frac{s}{2}+1} u\|_{L^2}\|\nabla u\|_{L^\infty}\right)+ \left\|\Lambda^{\frac{s}{2}+2}b\right\|_{L^2}^2\notag.
\end{align}

Similarly, one can use \eqref{equation}$_2$ and integrations by parts to estimate $J_2$ as
\begin{align}
	J_2&=-\int_{\mathbb{T}^n}\left(\tb\cdot\nabla u\right)\cdot \Lambda^s\left(\Delta b-u \cdot \nabla b+\tilde{b} \cdot \nabla u+b \cdot \nabla u\right)\dd x\notag\\
	\leq& -\int_{\mathbb{T}^n}\left(\tb\cdot\nabla u\right)\cdot \Lambda^s\Delta b\dd x- \int_{\mathbb{T}^n} \left(\tb\cdot\nabla u\right)\cdot \Lambda^s\left(\tb\cdot\nabla u\right)\dd x\label{J2}\\
&+\|\tb\cdot\nabla \Lambda^\frac{s}{2}u\|_{L^2}\|\Lambda^\frac{s}{2}\left(b \cdot \nabla u-u \cdot \nabla b\right)\|_{L^2}\notag\\
\leq&- \left\|\Lambda^{\frac{s}{2}}\left(\tb\cdot\nabla u\right)\right\|_{L^2}^2+\frac{1}{2}\left\|\Lambda^{\frac{s}{2}}\left(\tb\cdot\nabla u\right)\right\|_{L^2}^2+\frac{1}{2}\left\|\Lambda^{\frac{s}{2}+2}b\right\|_{L^2}^2\notag\\
&+\|\Lambda^{\frac{s}{2}+1} u\|_{L^2}\left(\|\Lambda^{\frac{s}{2}+1} b\|_{L^2}\|\nabla u\|_{L^\infty}+\|\Lambda^{\frac{s}{2}+1} u\|_{L^2}\|\nabla b\|_{L^\infty}\right).\notag
\end{align}

Finally, by substituting \eqref{J1} and \eqref{J2} into \eqref{4.1}, we can finish all the proof.
\end{proof}

\subsection{Proofs for the nonlinear stability results}\label{4/3}

\,\,\,\,\,\,\,\,Based on above analysis, we will concentrate on establishing {\it optimal} decay estimates in Sobolev space with {\it lower} regularity,
which helps to prevent any potential growth in the Sobolev norm of $u$ and $b$. Fortunately, we discover a {\it new} dissipative mechanism (see \eqref{4.4}), that brings us the {\it sharp} decay estimates
(see \eqref{decay}).
\begin{lemma}\label{lem4}
Let $m\geq 2r+\dfrac{n}{2}+3$,	assume that there exist $T>0$
	\begin{align}\label{small}
		\sup_{t\in[0,T]}\left(\|u(t)\|_{H^m}+\|b(t)\|_{H^m}\right)\leq \delta,
	\end{align}
for a sufficiently small $\delta$ and satisfies
	\begin{align}
		\int_{\mathbb{T}^n}u_0\dd x=\int_{\mathbb{T}^n}b_0\dd x=0.\label{meanzero}
	\end{align}
	Then it holds that, for any $0\leq t\leq T$,
	\begin{align}
		\|u(t)\|_{H^{\frac{s}{2}+1}}+\|b(t)\|_{H^{\frac{s}{2}+1}}\leq C(1+t)^{-\frac{m-\frac{s}{2}-1}{2+2r}},\label{decay}
	\end{align}
with $2r\leq s< n+2+2r$.
\end{lemma}
\begin{proof}
Owing to \eqref{meanzero}, one has for any $t\geq 0$,
\begin{equation}\label{meanzero3}
\int_{\mathbb{T}^n}u(x,t)\dd x=\int_{\mathbb{T}^n}b(x,t)\dd x=0,
\end{equation}
and as a result
\begin{equation}\label{meanzero2}
 \w{u}(0,t)=0,\quad  \w{b}(0,t)=0.
\end{equation}

Taking $\lambda=\frac{s}{2}+1$, $\mu=0$, $\nu=1$ and multiplying \eqref{energy} with $A=\frac{|\tilde{b}|}{2}+1$, then summing up the resultant and \eqref{key1}, it yields that
\begin{align}
	&\frac{1}{2}	\frac{\dd}{\dd t}\left(A\left(\|\Lambda^{\frac{s}{2}+1} u\|_{L^2}^2+\|\Lambda^{\frac{s}{2}+1} b\|_{L^2}^2\right)-\int_{\mathbb{T}^n}\left(\tb\cdot\nabla u\right)\cdot \Lambda^s b\dd x\right)\label{esti1}\\
	&\leq -\frac{1}{4}\left(\|\Lambda^{\frac{s}{2}+2} b\|_{L^2}^2+\left\|\Lambda^{\frac{s}{2}}\left(\tb\cdot\nabla u\right)\right\|_{L^2}^2\right)+C \left(\|\nabla u\|_{L^\infty}+\|\nabla b\|_{L^\infty}\right)\left(\|\Lambda^{\frac{s}{2}+1} u\|_{L^2}^2+\|\Lambda^{\frac{s}{2}+1} b\|_{L^2}^2\right)\notag\\
	&\leq -\frac{c^*}{4}\left(\|\Lambda^{\frac{s}{2}+1} b\|_{L^2}^2+\left\|\Lambda^{\frac{s}{2}-r}u\right\|_{L^2}^2\right)+C \left(\|\nabla u\|_{L^\infty}+\|\nabla b\|_{L^\infty}\right)\left(\|\Lambda^{\frac{s}{2}+1} u\|_{L^2}^2+\|\Lambda^{\frac{s}{2}+1} b\|_{L^2}^2\right).\notag
\end{align}
In the last step, we utilized Lemma \ref{2.1}, \eqref{meanzero2}, the fact $|k|^{s+4}>|k|^{s+2}\,\,(k\geq 1)$ and set $c^*=\min\{1,c\}$ with $c$ is the constant in \eqref{Diophantine}. Thanks to Lemma \ref{chazhi1}, we have
\begin{align}
	\|\nabla u\|_{L^\infty}\leq  C\|\Lambda^l u\|_{L^2}^{1-2\beta}\|\Lambda^{\frac{s}{2}-r}u\|_{L^2}^{2\beta},\label{inter1}
\end{align}
where $\beta\in [0,\frac{1}{2}]$ satisfies $$1=(l-\frac{n}{2})(1-2\beta)+(\frac{s}{2}-r-\frac{n}{2})2\beta\Rightarrow
	1+\frac{n}{2}=l(1-2\beta)+(\frac{s}{2}-r)2\beta,$$ with $2r\leq s< n+2+2r$, the upper bound of $s$ is due to satisfy $\frac{s}{2}-r>1+\frac{n}{2}.$ Similarly,
\begin{align}
	\|\Lambda^{\frac{s}{2}+1} u\|_{L^2}^2\leq C\|\Lambda^l u\|_{L^2}^{2\gamma}\|\Lambda^{\frac{s}{2}-r}u\|_{L^2}^{2-2\gamma},\label{inter2}
\end{align}
where $\gamma\in[0,1]$ such that $$\frac{s}{2}+1=l\gamma+(\frac{s}{2}-r)(1-\gamma)\Rightarrow s+2=2l\gamma+(\frac{s}{2}-r)(2-2\gamma).$$Let $\beta=\gamma$, then we have
$$s+\frac{n}{2}+3=l+s-2r \Rightarrow l=2r+\frac{n}{2}+3.$$ With the help of \eqref{inter1}, \eqref{inter2}, \eqref{meanzero2} and the fact $|k|^{s+2}>|k|^{s-2r}\,\,(k\geq 1)$, it yields that
\begin{align*}
	&C\left(\|\nabla u\|_{L^\infty}+\|\nabla b\|_{L^\infty}\right)\left(\|\Lambda^{\frac{s}{2}+1} u\|_{L^2}^2+\|\Lambda^{\frac{s}{2}+1} b\|_{L^2}^2\right)\\
	\leq&C (\|u(t)\|_{\dot{H}^l}+\|b(t)\|_{\dot{H}^l})\left(\left\|\Lambda^{\frac{s}{2}-r}u\right\|_{L^2}^2+\left\|\Lambda^{\frac{s}{2}-r}b\right\|_{L^2}^2\right)\\
	\leq&C (\|u(t)\|_{H^m}+\|b(t)\|_{H^m})\left(\left\|\Lambda^{\frac{s}{2}-r}u\right\|_{L^2}^2+\left\|\Lambda^{\frac{s}{2}+1}b\right\|_{L^2}^2\right),
\end{align*}
where $m\geq l=2r+\frac{n}{2}+3.$ Which can be updated as, by applying \eqref{small} and taking $\delta>0$ small enough
\begin{align}\notag
	&C\left(\|\nabla u\|_{L^\infty}+\|\nabla b\|_{L^\infty}\right)\left(\|\Lambda^{\frac{s}{2}+1} u\|_{L^2}^2+\|\Lambda^{\frac{s}{2}+1} b\|_{L^2}^2\right)\label{esti2}\\
\leq& \frac{c^*}{8}\left(\left\|\Lambda^{\frac{s}{2}-r}u\right\|_{L^2}^2+\left\|\Lambda^{\frac{s}{2}+1}b\right\|_{L^2}^2\right).
\end{align}
Thus, combining \eqref{esti1} and \eqref{esti2} leads to
	\begin{align}\notag
		&\frac{1}{2}	\frac{\dd}{\dd t}\left(A\left(\|\Lambda^{\frac{s}{2}+1} u\|_{L^2}^2+\|\Lambda^{\frac{s}{2}+1} b\|_{L^2}^2\right)-\int_{\mathbb{T}^n}\left(\tb\cdot\nabla u\right)\cdot \Lambda^s b\dd x\right)\label{4.2}\\
\leq& -\frac{c^*}{8}\left(\|\Lambda^{\frac{s}{2}+1} b\|_{L^2}^2+\left\|\Lambda^{\frac{s}{2}-r}u\right\|_{L^2}^2\right).
	\end{align}

Let $M$ be a number to be specified later such that $\dfrac{A}{M}\leq 1$ and by Plancherel’s theorem, we can conclude
\begin{align}\notag
	&\frac{A}{M}\|\Lambda^{\frac{s}{2}+1} u\|_{L^2}^2-\|\Lambda^{\frac{s}{2}-r} u\|_{L^2}^2=\sum_{|k|\neq 0}\left(\frac{A}{M}|k|^{s+2}-|k|^{s-2r}\right)\left|\w{u}(k)\right|^2\label{AM}\\
	\leq& \frac{A}{M}\sum_{\frac{A}{M}\geq |k|^{-2r-2}}|k|^{s+2}\left|\w{u}(k)\right|^2\notag\\
	\leq& \left(\frac{A}{M}\right)^{1+\frac{m-\frac{s}{2}-1}{1+r}}\sum_{\frac{A}{M}\geq |k|^{-2r-2}}\left(\dfrac{M}{A}\right)^{\frac{m-\frac{s}{2}-1}{1+r}}\left(\frac{1}{|k|}\right)^{2m-s-2}|k|^{2m}\left|\w{u}(k)\right|^2\notag\\
	\leq& \left(\frac{A}{M}\right)^{1+\frac{m-\frac{s}{2}-1}{1+r}}\|\Lambda^{m}u\|_{L^2},
\end{align}
the final step in the above equation is justified by the condition that $m\geq \frac{s}{2}+1$.
Utilizing the Holder inequality and the Poincaré inequality, we can derive
\begin{align}\label{4.3}
	\left|\int_{\mathbb{T}^n}\left(\tb\cdot\nabla u\right)\cdot \Lambda^s b\dd x\right|\leq|\tilde{b}| \|\Lambda^{\frac{s}{2}+1}u\|_{L^2}\|\Lambda^\frac{s}{2}b\|_{L^2}\leq \frac{|\tilde{b}|}{2}\|\Lambda^{\frac{s}{2}+1}u\|_{L^2}^2+\frac{|\tilde{b}|}{2}\|\Lambda^{\frac{s}{2}+1}b\|_{L^2}^2,
\end{align}
where we used \eqref{meanzero2} and the fact $|k|^{s+2}>|k|^{s}\,\,(k\geq 1)$ in \eqref{4.3}. Making use of \eqref{AM} and \eqref{4.3},
we can update \eqref{4.2} as
\begin{align}
&\frac{\dd}{\dd t}\left(A\left(\|\Lambda^{\frac{s}{2}+1} u\|_{L^2}^2+\|\Lambda^{\frac{s}{2}+1} b\|_{L^2}^2\right)-\int_{\mathbb{T}^n}\left(\tb\cdot\nabla u\right)\cdot \Lambda^s b\dd x\right)\label{4.4}\\
\leq& -\frac{c^*}{4M}A\left(\|\Lambda^{\frac{s}{2}+1} u\|_{L^2}^2+\|\Lambda^{\frac{s}{2}+1} b\|_{L^2}^2\right)+\frac{c^*}{4}\left(\frac{A}{M}\|\Lambda^{\frac{s}{2}+1} u\|_{L^2}^2-\|\Lambda^{\frac{s}{2}-r} u\|_{L^2}^2\right)\notag\\
\leq&  -\frac{c^*}{8M}\left(A\left(\|\Lambda^{\frac{s}{2}+1} u\|_{L^2}^2+\|\Lambda^{\frac{s}{2}+1} b\|_{L^2}^2\right)-\int_{\mathbb{T}^n}\left(\tb\cdot\nabla u\right)\cdot \Lambda^s b\dd x\right)+\frac{C}{4M^{1+\frac{m-\frac{s}{2}-1}{1+r}}}\|\Lambda^{m}u\|_{L^2}^2.\notag
\end{align}

In what follows, we will establish corresponding decay estimate. To this end, we recall the Lyapunov functional in \eqref{lya} as
\begin{align*}
	F(t)\triangleq A\left(\|\Lambda^{\frac{s}{2}+1} u(t)\|_{L^2}^2+\|\Lambda^{\frac{s}{2}+1} b(t)\|_{L^2}^2\right)-\int_{\mathbb{T}^n}\left(\tb\cdot\nabla u(t)\right)\cdot \Lambda^s b(t)\dd x,
\end{align*}
which satisfies
\begin{align*}
	F(t)\geq \left(\|\Lambda^{\frac{s}{2}+1} u\|_{L^2}^2+\|\Lambda^{\frac{s}{2}+1} b\|_{L^2}^2\right),
\end{align*}
according to \eqref{4.3}. Then, one can take $M=A+\dfrac{c^*t}{8j}$ and multiply both sides of \eqref{4.4} by $(A+\dfrac{c^*t}{8j})^j$ to obtain
\begin{align}
	\frac{\dd}{\dd t}\left((A+\dfrac{c^*t}{8j})^jF(t)\right)\leq C(A+\dfrac{c^*t}{8j})^{j-1-\frac{m-\frac{s}{2}-1}{1+r}}\|\Lambda^{m}u(t)\|_{L^2}^2,\label{4.5}
\end{align}
where $j>\dfrac{m-r-1}{1+r}\geq \dfrac{m-\frac{s}{2}-1}{1+r}$. Integrating \eqref{4.5} from $0$ to $t$ yields
\begin{align*}
	(A+\dfrac{c^*t}{8j})^jF(t)\leq A^jF(0)+C\int_{0}^t(A+\dfrac{c^*\tau}{8j})^{j-1-\frac{m-\frac{s}{2}-1}{1+r}}\|\Lambda^{m}u(\tau)\|_{L^2}^2\dd \tau,
\end{align*}
which further implies, after multiplying both sides with $(1+t)^{\frac{m-\frac{s}{2}-1}{1+r}-j}$, that
\begin{align*}
&(1+t)^{\frac{m-\frac{s}{2}-1}{1+r}}F(t)\\
\leq& C (1+t)^{\frac{m-\frac{s}{2}-1}{1+r}-j}A^j F(0)+C(1+t)^{\frac{m-\frac{s}{2}-1}{1+r}-j}\int_{0}^t(A+\dfrac{c^*\tau}{8j})^{j-1-\frac{m-\frac{s}{2}-1}{1+r}}\|\Lambda^{m}u(\tau)\|_{L^2}^2\dd \tau\\
\leq& C+C(1+t)^{\frac{m-\frac{s}{2}-1}{1+r}-j+j-\frac{m-\frac{s}{2}-1}{1+r}}\sup_{0\leq\tau\leq t}\|\Lambda^{m}u(\tau)\|_{L^2}^2\leq C,
\end{align*}
and thus
\begin{align}
	\|\Lambda^{\frac{s}{2}+1} u\|_{L^2}^2+\|\Lambda^{\frac{s}{2}+1} b\|_{L^2}^2\leq C(1+t)^{-\frac{m-\frac{s}{2}-1}{1+r}}.\label{Hsdecay}
\end{align}
In the end, by combing \eqref{Hsdecay}, the fact $(1+|k|^2)^{\frac{s}{2}+1}\leq C|k|^{s+2}$ for $|k|\geq 1$ and Plancherel formula, we obtain \eqref{decay} and finish the proof.
\end{proof}

We now prove Theorem \ref{thm}. The local well-posedness part of system \eqref{equation} can be shown via a rather standard procedure such as Friedrichs' method of cutoff in Fourier space,
see \cite{li-2017-local,fefferman-2014-local} for example. Therefore, it suffices to establish the global a priori estimates of solutions in $H^m$ and then use the continuity method to show our main result.

\begin{proof}[Proof of Theorem \ref{thm}]
Based on the above analysis, there exists a positive time $T>0$ and a unique local solution $(u, b) \in C( [0,T]; H^m )$ of the system \eqref{equation}. Moreover, according to the initial assumptions \eqref{smallcondition}, we can assume that
	\begin{align*}
	\sup_{t\in[0,T]}\left(\|u(t)\|_{H^m}+\|b(t)\|_{H^m}\right)\leq \delta,
\end{align*}
for some $0<\delta<1$.

Thanks to Lemma \ref{lem4}, by setting $s$ to be equal to $2r$, for $r>n-1$, it holds that
		\begin{align}\label{lower}
		\|u(t)\|_{H^{r+1}}+\|b(t)\|_{H^{r+1}}\leq C(1+t)^{-\frac{m-r-1}{2+2r}},
	\end{align}
which together with imbedding inequality and $r+1>n\geq \frac{n}{2}+1$ implies
\begin{align}
\|\nabla u\|_{L^\infty}+\|\nabla b\|_{L^\infty}\leq C	\left(\|u(t)\|_{H^{r+1}}+\|b(t)\|_{H^{r+1}}\right)\leq C(1+t)^{-\frac{m-r-1}{2+2r}}.\label{linfinity}
\end{align}
\eqref{linfinity} and Lemma \ref{energy} with $\lambda=m$ lead to
	\begin{align}
	\frac{1}{2}\frac{\dd}{\dd t}\left(\|\Lambda^m u\|_{L^2}^2+\|\Lambda^m b\|_{L^2}^2\right)
	\leq C(1+t)^{-\frac{m-r-1}{2+2r}}\left(\|\Lambda^m u\|_{L^2}^2+\|\Lambda^m  b\|_{L^2}^2\right).\label{nablam}
\end{align}

Due to $m>3r+3\geq 2r+\frac{n}{2}+3$, then $\frac{m-r-1}{2+2r}>1$. Applying Gronwall's inequality on \eqref{nablam}, it follows that
\begin{align*}
	\|\Lambda^m u\|_{L^2}^2+\|\Lambda^m b\|_{L^2}^2\leq C\|\Lambda^m u(0)\|_{L^2}^2+\|\Lambda^m b(0)\|_{L^2}^2\leq C\varepsilon^2,
\end{align*}
which further implies, after using \eqref{meanzero3} and Poincaré inequality, that
\begin{align}
	\|u(t)\|_{H^m}^2+\| b(t)\|_{H^m}^2\leq  C\varepsilon^2,\quad t\in[0,T].\label{Hm}
\end{align}
With \eqref{Hm} at hand, by choosing $\varepsilon$ be sufficiently small such that $C\varepsilon < \frac{1}{2}\delta$, then we can conclude from the continuity argument that the local solution can be extended to a global one, and the following estimate holds
\begin{align}\label{hign}
		\|u(t)\|_{H^m}^2+\| b(t)\|_{H^m}^2\leq  C\varepsilon,~ t\in[0,\infty).
\end{align}

In what follows, it suffices to establish the decay estimates of solution in $H^\alpha$ for any $r+1\leq \alpha\leq m$. In fact, by interpolation inequality,
we can get from \eqref{lower} and \eqref{hign}, that
\begin{align*}
		\|u(t)\|_{H^\alpha}+\|b(t)\|_{H^\alpha}&\leq C\|u(t)\|_{H^{r+1}}^{\frac{m-\alpha}{m-r-1}}\|u(t)\|_{H^m}^{\frac{\alpha-r-1}{m-r-1}}
+C\|b(t)\|_{H^{r+1}}^{\frac{m-\alpha}{m-r-1}}\|b(t)\|_{H^m}^{\frac{\alpha-r-1}{m-r-1}}\\
&\leq C(1+t)^{-\frac{m-\alpha}{2+2r}}.
\end{align*}
Thus, we finish the proof of Theorem \ref{thm}.
\end{proof}


\section*{Ackonwledgments}
\,\,\,\,\,\,\,\,Quansen Jiu was partially supported by National Natural Science Foundation of China under grants (No.\,\,11931010, No.\,\,12061003). Jitao Liu was partially supported by National Natural Science Foundation of China under grants (No.\,\,11801018, No.\,\,12061003), Beijing Natural Science Foundation under grant (No.\,\,1192001) and Beijing University of Technology under grant (No.\,\,006000514123513).


\begin{thebibliography}{00}




	\bibitem{abidi-2017-GlobalSolution3D}
H. Abidi, P. Zhang, {\em On the {{global solution}} of a 3-{{D MHD
			system}} with {{initial data}} near {{equilibrium}}}, Comm. Pure Appl. Math. 70 (2017), no. 8, 1509--1561.

\bibitem{Alinhac}
S. Alinhac, P. Gérard, {\em Pseudo-differential operators and the Nash-Moser theorem}, Graduate Studies in Mathematics, 82. American Mathematical Society, Providence, RI, 2007.

\bibitem{brezis-2019-gag}
H. Brezis, P. Mironescu, {\em Where Sobolev interacts with Gagliardo–Nirenberg}, J. Funct. Anal. 277 (2019), no. 8, 2839--2864.

\bibitem{chen-2022-3dmhd-Diophant}
W. Chen, Z. Zhang, J. Zhou, {\em Global well-posedness for the 3-D MHD equations with partial diffusion in the periodic domain}, Sci. China Math. 65 (2022), no. 2, 309--318.

\bibitem{duvaut-1972}
G. Duvaut, J. L. Lions, {\em Inéquations en thermoélasticité et magnétohydrodynamique}, Arch. Rational Mech. Anal. 46 (1972), 241--279.

\bibitem{fefferman-2014-local}
C. L. Fefferman, D. S. McCormick, J. C. Robinson, J. L. Rodrigo, {\em Higher order commutator estimates and local existence for the non-resistive MHD equations and related models}, J. Funct. Anal. 267 (2014), no. 4, 1035--1056.

 \bibitem{jiang-2021-Asymptotic behaviors}
F. Jiang, S. Jiang, {\em Asymptotic behaviors of global solutions to the two-dimensional non-resistive MHD equations with large initial perturbations}, Adv. Math. 393 (2021), Paper No. 108084, 79 pp.

\bibitem{jiang-2023-arma-nonresistive}
F. Jiang, S. Jiang, {\em On magnetic inhibition theory in 3D non-resistive magnetohydrodynamic fluids: global existence of large solutions}, Arch. Ration. Mech. Anal. 247 (2023), no. 5, Paper No. 96, 35 pp.

\bibitem{kato-1988-CommutatorEstimatesEuler}
T. Kato, G. Ponce, {\em Commutator estimates and the Euler and Navier-Stokes equations}, Comm. Pure Appl. Math. 41 (1988), no. 7, 891--907.

\bibitem{kenig}
 C. Kenig, G. Ponce, L. Vega, {\em Well-posedness of the initial value problem for the Korteweg-de Vries equation},  J. Amer. Math. Soc. 4 (1991), no. 2, 323--347.


\bibitem{li-2017-local}
J. Li, W. Tan, Z. Yin, {\em Local existence and uniqueness for the non-resistive MHD equations in homogeneous Besov spaces}, Adv. Math. 317 (2017), 786--798.

\bibitem{lizhai-2023-jga}
Y. Li, H. Xu, X. Zhai, {\em Global smooth solutions to the 3D compressible viscous non-isentropic magnetohydrodynamic flows without magnetic diffusion},  J. Geom. Anal. 33 (2023), no. 8, Paper No. 246, 32 pp.

	\bibitem{lin-2015-GlobalSmallSolutions}
F. Lin, L. Xu, P. Zhang, {\em Global small solutions of 2-D incompressible MHD system}, J. Differential Equations 259 (2015), no. 10, 5440--5485.

	\bibitem{lin-2014-GlobalSmallSolutions}
F. Lin, P. Zhang, {\em Global small solutions to an MHD-type system: the three-dimensional case}, Comm. Pure Appl. Math. 67 (2014), no. 4, 531--580.


\bibitem{panGlobalClassicalSolutions2018}
R. Pan, Y. Zhou, Y. Zhu, {\em Global classical solutions of three dimensional viscous MHD system without magnetic diffusion on periodic boxes}, Arch. Ration. Mech. Anal. 227 (2018), no. 2, 637--662.


\bibitem{ren2014global}
X. Ren, J. Wu, Z. Xiang, Z. Zhang, {\em Global existence and decay of
	smooth solution for the 2-D MHD equations without magnetic diffusion},
J. Funct. Anal. 267 (2014), no. 2, 503--541.

\bibitem{roberts1967introduction}
P. H. Roberts, {\em An introduction to magnetohydrodynamics}, Longmans, London, 1967.

\bibitem{priest-2000}
E. Priest, T. Forbes, {\em Magnetic Reconnection, MHD Theory and Applications}, Cambridge University Press, Cambridge, 2000.

\bibitem{sermange-1983}
M. Sermange, R. Temam, {\em Some mathematical questions related to the MHD equations}, Comm. Pure Appl. Math. 36 (1983), no. 5, 635--664.

\bibitem{zhang-2020-GlobalWellPosedness2D}
D. Wei, Z. Zhang, {\em Global well-posedness for the 2-D MHD equations with magnetic diffusion}, Commun. Math. Res. 36 (2020), no. 4, 377--389.
		
\bibitem{WuZhai}
J. Wu, X. Zhai, {\em Global small solutions to the 3D compressible viscous non-resistive MHD system}, Math. Models Methods Appl. Sci. 33 (2023), no. 13, 2629--2656.


\bibitem{xu-2015-GlobalSmallSolutions}
L. Xu, P. Zhang, {\em Global small solutions to three-dimensional incompressible magnetohydrodynamical system}, SIAM J. Math. Anal. 47 (2015), no. 1, 26--65.

\bibitem{ye-2022-GlobalWellposednessNonviscous}
W. Ye, Z. Yin, {\em Global well-posedness for the non-viscous MHD equations with magnetic diffusion in critical Besov spaces}, Acta Math. Sin. (Engl. Ser.) 38 (2022), no. 9, 1493--1511.

\bibitem{zhai-2023-2dmhdstability-Diophant}
X. Zhai, {\em Stability for the 2D incompressible MHD equations with only magnetic diffusion}, J. Differential Equations 374 (2023), 267--278.

\bibitem{zhou-2018-GlobalClassicalSolutions}
Y. Zhou, Y. Zhu, {\em Global classical solutions of 2D MHD system with only magnetic diffusion on periodic domain}, J. Math. Phys. 59 (2018), no. 8, 081505, 12 pp.

\bibitem{zhangting}
T. Zhang, {\em Global solutions to the 2D viscous, non-resistive MHD system with large background magnetic field}, J. Differential Equations 260 (2016), no. 6, 5450--5480.


\end{thebibliography}
\end{document}